\documentclass[11pt]{amsart}

\usepackage[left=1in, right=1in, top=1.2in, bottom=1.2in]{geometry}
\usepackage{microtype, mathtools, mathrsfs, enumerate, dsfont, esvect}
\usepackage{verbatim}
\usepackage[linktocpage]{hyperref}
\hypersetup{
    colorlinks=true,        
    linkcolor=red,          
    citecolor=blue,         
    filecolor=magenta,      
    urlcolor=cyan           
}
\usepackage[all]{xy}

\usepackage{amssymb}
\usepackage{url}

\newtheorem{theorem}{Theorem}[section]

\newtheorem{prop}[theorem]{Proposition}
\newtheorem{lemma}[theorem]{Lemma}
\newtheorem{definition-lemma}[theorem]{Definition-Lemma}

\newtheorem{corollary}[theorem]{Corollary}

\theoremstyle{definition}

\newtheorem{definition}[theorem]{Definition}

\theoremstyle{plain}
\numberwithin{equation}{section}

\theoremstyle{remark}

\newtheorem{remark}[theorem]{Remark}
\newtheorem{example}[theorem]{Example}
\newtheorem{remarks}[theorem]{Remarks}

\DeclareMathAlphabet{\mathpzc}{OT1}{pzc}{m}{it} 

\DeclareMathOperator{\chara}{char}
\DeclareMathOperator{\gr}{gr} 

\DeclareMathOperator{\Tor}{Tor}

\DeclareMathOperator{\coh}{H}
 
\DeclareMathOperator{\HH}{HH}

\DeclareMathOperator{\Ext}{Ext}
\DeclareMathOperator{\Hom}{Hom}

\DeclareMathOperator{\Tot}{Tot}

\newcommand{\ot}{\otimes}

\newcommand{\DOT}{\setlength{\unitlength}{1pt}\begin{picture}(2.5,2)
               (1,1)\put(2,2.5){\circle*{2}}\end{picture}}

\newcommand{\bu}{\DOT}

\begin{document}



\thanks{The second author was partially supported by Simons Collaboration Grant No. 688403. The third author was partially supported by NSF grant DMS-1665286.}

\title[New approaches to finite generation of cohomology rings]
{New approaches to finite generation \\ of cohomology rings}

\author{Van C.\ Nguyen} 
\address{Department of Mathematics\\United States Naval Academy\\Annapolis, MD 21402, U.S.A.}
\email{vnguyen@usna.edu}

\author{Xingting Wang}
\address{Department of Mathematics, Louisiana State University, Baton Rouge, LA 70803, USA}
\email{xingting.wang@math.lsu.edu}

\author{Sarah Witherspoon}
\address{Department of Mathematics\\Texas A\&M University \\College Station, TX 77843, U.S.A.}
\email{sjw@math.tamu.edu}

\date{\today}
 

\begin{abstract}
In support variety theory, representations of a finite dimensional (Hopf) algebra $A$ can be studied geometrically by associating any representation of $A$ to an algebraic variety using the cohomology ring of $A$. An essential assumption in this theory is the finite generation condition for the cohomology ring of $A$ and that for the corresponding modules. 

In this paper, we introduce various approaches to study the finite generation condition. First, for any finite dimensional Hopf algebra $A$, we show that the finite generation condition on $A$-modules can be replaced by a condition on any affine commutative $A$-module algebra $R$ under the assumption that $R$ is integral over its invariant subring $R^A$. Next, we use a spectral sequence argument to show that a finite generation condition holds for certain filtered, smash and crossed product algebras in positive characteristic if the related spectral sequences collapse. Finally, if $A$ is defined over a number field over the rationals, we construct another finite dimensional Hopf algebra $A'$ over a finite field, where $A$ can be viewed as a deformation of $A'$, and prove that if the finite generation condition holds for $A'$, then the same condition holds for $A$. 
\end{abstract}

\maketitle

\section{Introduction}

Hochschild cohomology was introduced by Hochschild in 1945 \cite{GHoch59} for any associative algebra. Gerstenhaber \cite{Ger63} showed that this cohomology has a graded algebra structure (via cup product) and a graded Lie algebra structure (via a Lie bracket or Gerstenhaber bracket). These two algebraic structures are compatible in such a way that makes Hochschild cohomology a {\em Gerstenhaber algebra}. Many mathematicians have since investigated Hochschild cohomology $\HH^*(A)$ for various types of algebras $A$, and it has been useful in many settings, including algebraic deformation theory (e.g.,~\cite{Ger64}) and support variety theory (e.g.,~\cite{EHSS, SO}). 

Generally speaking, the theory of support varieties studies representations of an algebra $A$ geometrically by associating each finitely generated $A$-module $M$ to a certain algebraic variety $V(M)$, namely the variety of the kernel of the graded ring homomorphism $-\otimes_AM: \HH^*(A) \to \Ext^*_A(M,M)$. In support variety theory, the following finite generation assumption on the Hochschild cohomology ring $\HH^*(A)$ of $A$ and on bimodules is essential (see e.g.,~\cite{EHSS, SO, Solberg}).

\vspace{0.5em}
\begin{itemize}
\item[\textbf{(fg)}:] $\HH^*(A,M)$ is a noetherian module over $\HH^*(A)$, for any finitely generated $A$-bimodule $M$.  
\end{itemize}
\vspace{0.5em}

\noindent
In particular, condition {\rm \textbf{(fg)}} implies that the Hochschild cohomology ring $\HH^*(A)$ is finitely generated. Note that this is not always true for a finite dimensional algebra $A$; see Xu's counterexample in \cite{Xu}. We wish to know more finite dimensional algebras $A$ that satisfy condition {\rm \textbf{(fg)}}. There are several conditions equivalent to {\rm \textbf{(fg)}}, which are sometimes more convenient to use. For completeness, we include them in Proposition~\ref{equivfg} below.

For any finite dimensional Hopf algebra $A$ over a field $k$, one can also study the support variety theory using the cohomology ring $\coh^*(A,k)$. There are certain classes of Hopf algebras $A$ whose cohomology rings are known to be finitely generated, but it is still unknown in general if this cohomology ring is always finitely generated. In fact, this is a long-standing conjecture, which is formulated in the setting of finite tensor categories by Etingof and Ostrik in \cite[Conjecture 2.18]{EsOt}. If this conjecture holds true for $A$, one can define support varieties over $A$ by using a graded ring homomorphism $ - \ot_k M: \coh^*(A,k) \to \Ext_A^*(M,M)$ similar to that described for Hochschild cohomology above. This homomorphism factors through the action of $\HH^*(A)$ on $\Ext^*_A(M,M)$~\cite{PW}, giving a connection between the support variety theories defined via Hopf algebra cohomology $\coh^*(A,k)$ and Hochschild cohomology $\HH^*(A)$. We will work specifically with the following noetherian assumption of a finite dimensional Hopf algebra $A$ over a field $k$ (see e.g.,~\cite{FW2015}), and more generally for any finite dimensional augmented $k$-algebra $A$:

\vspace{0.5em}
\begin{itemize}
\item[\textbf{(hfg)}:] $\coh^*(A,M)$ is a noetherian module over $\coh^*(A,k)$, for any finite dimensional $A$-module $M$.
\end{itemize}
\vspace{0.5em}

\noindent
In particular, when $A$ is a finite dimensional Hopf algebra, condition \textbf{(hfg)} implies that $\coh^*(A,k)$ is finitely generated, which is the finite generation conjecture mentioned above. We wish to identify more finite dimensional Hopf algebras $A$ that satisfy condition \textbf{(hfg)}, and hence move forwards towards the goal of proving the finite generation conjecture. \\

\noindent
\underline{{\bf Main Results:}} We take three different approaches in this finite generation problem: \vspace{0.1in}

Our first result in Section~\ref{sec:fg conditions} provides several equivalent assumptions for \textbf{(hfg)} including the finite generation condition on the cohomology ring $\coh^*(A,R)$ for any affine commutative $A$-module algebra $R$. Our result in Lemma~\ref{lem:hfg} shows that these two assumptions are equivalent if $A$ has the integral property (see Section~\ref{subsec:integral}). A classical result in algebraic groups says that any finite dimensional cocommutative Hopf algebra has the integral property. Zhu \cite{Zhu} proved that the integral property holds for any finite dimensional cosemisimple Hopf algebra and Skryabin \cite{Sky} showed that any finite dimensional Hopf algebra in positive characteristic has the integral property. So our result can be applied to all these cases.

Next, in Section~\ref{sec:homPI}, we study the finite generation conditions that are preserved under certain spectral sequences related to filtered, smash and crossed product algebras (see Theorems \ref{thm:FFGC}, \ref{thm:FiniteTypeCoh}, and \ref{thm:FiniteTypeHochschild} and see the Appendix for those we use here). An answer to the finite generation conditions on multiplicative spectral sequences relies on having suitable \emph{permanent cocycles} (those universal cocycles which survive under the differential maps of all pages in a cohomology spectral sequence), see e.g., \cite{FS,MPSW,NWW2019,NWi}. In this work, we use the Frobenius map to construct such permanent cocycles over the field of positive characteristic (see Lemma \ref{lem:key}). In Propositions \ref{prop:hfg} and \ref{prop:fg}, we employ the Hilton-Eckmann argument in a suspended monoidal category to show that the cohomology ring $\coh^*(A,R)$ of a finite dimensional Hopf algebra $A$ with coefficients in an $A$-module algebra $R$ is finite over its affine graded center under some assumptions on $R$.  We apply this result to the first pages of  May spectral sequences related to filtered algebras and to the second pages of Lyndon-Hochschild-Serre spectral sequences related to smash and crossed products. Moreover, if the related spectral sequences collapse over a field of positive characteristic, then we are able to construct permanent cocycles by using the Frobenius map on their graded centers to conclude the finite generation conditions.

Finally, in Section~\ref{sec:mod p}, we apply the well-known reduction modulo $p$ method in number theory to support variety theory. For a large family of finite dimensional complex Hopf algebras $A$ (e.g., Lusztig's small quantum groups \cite{Lus90}), we can select a ``good" prime $p$ to construct another finite dimensional Hopf algebra $A'$ over the finite field $F_{q}$, where $q=p^n$ for some integer $n$. Here, $A$ can be viewed as a deformation of $A'$. We show that, under some conditions, if the newly constructed Hopf algebra $A'$ over the finite field satisfies \textbf{(hfg)} then so does the original complex Hopf algebra $A$. Namely, the finite generation conditions in positive characteristic can be lifted up to those in zero characteristic via reduction modulo $p$ (see Theorem \ref{thm:m1}). Moreover, our approach is applicable both to the cohomology ring of finite dimensional Hopf algebras and to the Hochschild cohomology of finite dimensional associative algebras (see Theorem \ref{thm:m2}). Therefore, our lifting method provides a viable connection between a field of \emph{characteristic zero} and a field of \emph{positive characteristic} in studying the cohomology of Hopf algebras. This link is often lacking when one classifies Hopf algebras, as the classification methods are quite different from one field to the other (see e.g., \cite{A, NWa, NWW2015, NWW2018, Wang2013}).


\section{Preliminaries}
\label{sec:prelim}

\numberwithin{equation}{subsection}

Throughout the paper, let $k$ be a base field. All modules are left modules and $\ot = \ot_k$ unless stated otherwise. We first recall the two cohomology types discussed in this paper and present some background material necessary to build up our main results. 

\begin{definition}
\label{def:hochschild}
Let $A$ be an algebra over $k$ and $M$ be an $A$-bimodule, which can be considered as a left module over the enveloping algebra $A^e = A \ot A^{op}$ of $A$. The {\bf Hochschild cohomology} of $A$ with coefficients in $M$ is
$$\HH^*(A,M)~:=~\bigoplus_{n \ge 0}\, \Ext^n_{A^e}(A,M),$$
where $A$ is an $A$-bimodule via the left and right multiplications in $A$.
\end{definition}

Under the cup product, $\HH^*(A)~:= ~\HH^*(A,A)$ is a graded commutative $k$-algebra and $\HH^*(A,M)$ is a module over $\HH^*(A)$. 

\begin{definition}
\label{def:coh}
Let $A$ be an augmented $k$-algebra and $M$ be a left $A$-module. The {\bf cohomology} of $A$ with coefficients in $M$ is
$$\coh^*(A,M):=  \bigoplus_{n \ge 0}\, \Ext^n_A (k,M),$$ 
where $k$ is an $A$-module via the augmentation map. 
\end{definition}

Under the Yoneda product, $\coh^*(A,k)$ is a graded $k$-algebra and $\coh^*(A,M)$ is a module over $\coh^*(A,k)$. If, in addition, $A$ is a finite dimensional Hopf algebra, then $\coh^*(A,k)$ is graded commutative (see also e.g.,~\cite{ginzburg-kumar93,Suarez-Alvarez}).

\subsection{Finite generation condition on Hochschild cohomology}
We consider the following finite generation conditions on Hochschild cohomology as
alternatives to the {\textbf{(fg)}} condition:
\vspace{0.5em}
\begin{itemize}
\item[\textbf{(fg1)}:] There is a finitely generated commutative graded subalgebra $H$ of $\HH^*(A)$ with degree-$0$ component $H_0=\HH^0(A)$; and
\item[\textbf{(fg2)}:] $\Ext^*_A(M,N)$ is a finitely generated $H$-module, for all pairs of finitely generated $A$-modules $M$ and $N$. 
\end{itemize}
\vspace{0.5em}

We next recall how these finite generation conditions and others are all equivalent.

\begin{prop}
\label{equivfg}
For any finite dimensional $k$-algebra $A$, the following are equivalent.
\begin{enumerate}
\item $\HH^*(A,M)$ is a noetherian module over $\HH^*(A)$, for any finite $A$-bimodule $M$, that is, $A$ satisfies {\rm \textbf{(fg)}}. 
\item $\HH^*(A)$ is a finitely generated algebra and $\Ext^*_A(A/\rm{rad}(A),A/\rm{rad}(A))$ is a finitely generated module over $\HH^*(A)$, where $\rm{rad}(A)$ denotes the radical of $A$. 
\item $\HH^*(A)$ is a finitely generated algebra and $\Ext^*_A(M,N)$ is a finitely generated module over $\HH^*(A)$, for all pairs of finite $A$-modules $M$ and $N$.
\item $A$ satisfies {\rm \textbf{(fg1)}} and {\rm \textbf{(fg2)}}.
\end{enumerate}
\end{prop}

\begin{proof}
(i)$\Leftrightarrow$(ii)$\Leftrightarrow$(iii): The proof is essentially that of \cite[Proposition 2.4]{EHSS}. Since $\HH^*(A)$ is graded commutative, we know it is finitely generated if and only if it is noetherian by \cite[Proposition~\ref{prop:1}(2)]{EHSS}. Also, a module over a noetherian ring is finitely generated if and only if it is a noetherian module. 

(ii)$\Leftrightarrow$(iv): See~\cite[Proposition 5.7]{Solberg}.
\end{proof}


\subsection{Properties of graded algebras}

Both cohomology structures $\HH^*(A)$ and $\coh^*(A,k)$ are graded algebras, so to work with these cohomology structures, we recall some basic properties of nonnegatively graded algebras. 

Let $A$ be a (nonnegatively) graded $k$-algebra, that is, $A=\bigoplus_{i\ge 0} A_i$ with $1_A \in A_0$ and $A_iA_j \subseteq A_{i+j}$ for all integers $i,j \geq 0$. If $A_0=k$, then $A$ is called {\bf connected graded}. Recall that $A$ is called {\bf $A_0$-affine} if it is a finitely generated algebra over the degree-$0$ component $A_0$. The following well-known results will be used many times throughout the paper without explicit citation. 

\begin{prop}
\label{prop:1}
Let $A=\bigoplus_{i\ge 0} A_i$ be a graded algebra for which the subalgebra $A_0$ is 
noetherian commutative. Then: 
\begin{enumerate}
\item $A$ is noetherian if and only if $A$ is graded noetherian.
\item Suppose $A$ is graded commutative. Then the following are equivalent:
  \begin{itemize}
    \item [(a)] $A$ is noetherian.
    \item [(b)] $A$ is $A_0$-affine.    
    \item [(c)] $A$ is a finite module over $A^{{\rm ev}}$ and $A^{{\rm ev}}$ is $A_0$-affine, where $A^{{\rm ev}}$ consists of all even-degree elements of $A$. 
  \end{itemize}  
\item Suppose $A$ is noetherian and is a finite module over some graded central subalgebra $Z$ with $Z_0=A_0$. Then $Z$ is $Z_0$-affine and noetherian.
\end{enumerate}
\end{prop}

\begin{proof}
For a proof of (i), see for example~\cite{Evens}.
For (ii), see~\cite[Lemma 3.2]{JL}.
For (iii), we can adapt the Artin-Tate Lemma \cite[Lemma 13.9.10]{MR} to the graded setting.
\end{proof}


\subsection{Integral property of Hopf algebras}
\label{subsec:integral}

We recall that a commutative ring $R$ is {\bf integral} over some subring $T$ if every element of $R$ is a root of some monic polynomial with coefficients in $T$. Furthermore, if $R$ is $k$-affine, then $R$ is integral over $T$ if and only if $R$ is a finite module over $T$. We say a Hopf algebra $A$ has the {\bf integral property} if
for any commutative $A$-module algebra $R$ (that is, $R$ is an $A$-module with unit and multiplication maps both $A$-module maps), it is integral over its invariant subring $R^A$. Note that, in characteristic zero, Zhu \cite{Zhu} showed that the four dimensional Sweedler Hopf algebra does {\em not} have the integral property. 

\begin{prop}
\label{prop:integral}
Let $A$ be a finite dimensional Hopf algebra over a field $k$. In each of the following cases, $A$ satisfies the integral property:
\begin{enumerate}
\item the base field $k$ has positive characteristic,
\item $A$ is semisimple,
\item $A$ is cosemisimple,
\item $A$ is cocommutative. 
\end{enumerate}
\end{prop}

\begin{proof}
(i) is \cite[Proposition 2.7 and the remark below it]{Sky}. (ii) is \cite[Theorem 6.1 and the remark below it]{Sky}. (iii) follows from \cite[Theorem 2.1]{Zhu}. (iv) is \cite[Theorem 4.2.1]{MO93}.
\end{proof}


\subsection{Extensions for the finite generation conditions} 
\label{subsec:FG}

In this section, we remark that the finite generation conditions {\rm \textbf{(fg)}} and {\rm \textbf{(hfg)}} that were defined on (Hochschild) cohomology of an algebra $A$ in the introduction can be extended to another algebra $B$ under some assumptions: 

\begin{lemma}\label{lem:FGCSUR}
Let $f: A\to B$ be a map of finite dimensional augmented $k$-algebras such that $B$ is projective as a (right) $A$-module via $f$. If $B$ satisfies {\rm \textbf{(hfg)}}, then so does $A$. 
\end{lemma}

\begin{proof}
Let $M$ be a finite dimensional $A$-module.
The coinduction functor $\Hom_A( _{A}B , - )$ is right adjoint to restriction from $B$ to $A$ along $f$.
Since $B$ is projective as an $A$-module, a projective
resolution of $k$ as a $B$-module restricts to a projective resolution of $k$ as an $A$-module.
Thus there is an isomorphism
\[
    \Ext^*_B(k, \Hom_A( _{A}B, M)) 
    \cong  \Ext^*_A(k, M).
\]
Under this isomorphism and restriction from
$\Ext^*_B(k,k)$ to $\Ext^*_A(k,k)$, 
the action of $\Ext^*_B(k,k)$ on $\Ext^*_B(k,\Hom_A( _{A}B,M))$ restricts to the action of $\Ext^*_A(k,k)$ on $\Ext^*_A(k,M)$. Under the assumption that $B$ satisfies {\bf (hfg)}, $\Ext^*_B(k,\Hom_A ( _{A}B,M))$ is a noetherian module over $\Ext^*_B(k,k)$, and so $\Ext^*_A(k,M)$ is a noetherian module over $\Ext^*_A(k,k)$.
\end{proof}

\begin{remark}
The same conclusion does not hold for {\bf(fg)}. For example, take $A=B \times C$ where $B$ satisfies condition {\bf(fg)} but $C$ does not. Then $B$ is projective as a left and right module over $A$ via the natural projection $f: A \to B$. Here $B$ satisfies {\bf(fg)} but $A$ does not.
\end{remark}

The next result was first proved in \cite{NP2018} in the context of finite tensor categories with respect to surjective tensor functors. We provide here an alternate proof in the special case of categories of modules over Hopf algebras. 

\begin{prop}
\label{prop:ext}
Let $A\subset B$ be an extension of finite dimensional Hopf algebras over a field $k$. If $B$ satisfies {\rm \textbf{(hfg)}}, then so does $A$. 
\end{prop}

\begin{proof}
By the Nichols-Zoeller freeness theorem \cite{NZ89}, every finite dimensional Hopf algebra is free as a module over each of its Hopf subalgebras. The statement now follows from Lemma \ref{lem:FGCSUR}.
\end{proof} 

Next we discuss the relations between conditions {\rm \textbf{(fg)}} and {\rm \textbf{(hfg)}} for a finite dimensional Hopf algebra $A$. Let $S$ denote the antipode map of $A$.

\begin{prop}\label{adjoint}
Let $A$ be a finite dimensional Hopf algebra over a field $k$. Then $A$ satisfies {\rm \textbf{(hfg)}} if and only if $A$ satisfies {\rm \textbf{(fg)}} and $\HH^*(A)$ is a finitely generated module over $\coh^*(A,k)$.
\end{prop}
\begin{proof}
Denote by $\text{Mod}(A)$ and $\text{Mod}(A^e)$ the categories of left $A$-modules and $A$-bimodules, respectively. There is a natural pair of adjoint functors 
$$\mathcal F=A^e\otimes_A-:\text{Mod}(A)\to \text{Mod}(A^e) \quad \text{ and } \quad \mathcal G=\Hom_{A^e}(\!_{A^e}A^e_A,-):\text{Mod}(A^e)\to \text{Mod}(A),
$$
where $A^e$ is viewed as a right $A$-module via the embedding $\delta: A\rightarrow A^e$ such that $\delta(a) = \sum a_1\ot S(a_2)$ for any $a\in A$. Note that $A^e$ is a free right $A$-module by the fundamental theory of Hopf modules (e.g., \cite[Lemma 2.3(ii)]{LOYW}). So $\mathcal F$ and $\mathcal G$ are both exact and preserve projective and injective objects, respectively. And there is an isomorphism of $A^e$-modules $A \cong A^e\ot_A k=\mathcal F(k)$~\cite[Lemma 7.1]{PW}.

Suppose $A$ satisfies {\rm \textbf{(hfg)}}. Apply \cite[Proposition 9.1(c)]{ESW} where $C=k$, $\mathcal F(k)=A$ and $D=M$ as any finite $A$-bimodule. Then $\HH^*(A,M)$ is finitely generated over $\HH^*(A)$ since $\coh^*(A, \mathcal G(M))$ is finitely generated over $\coh^*(A,k)$ by {\rm \textbf{(hfg)}}. Now letting $D=A$, Hochschild cohomology $\HH^*(A)=\coh^*(A, \mathcal G(A))$ is finitely generated over $\coh^*(A,k)$. Moreover, Proposition \cite[Proposition 9.1(d)]{ESW} implies that $\HH^*(A)$ is noetherian. 

Conversely, we can view any finite $A$-module $M$ as a bimodule over $A$ by equipping it with the trivial right action, which is denoted by $M^{\rm tr}$. It is clear that $\mathcal F(M^{\rm tr})=M$. By {\rm \textbf{(fg)}}, we know $\HH^*(A,M^{\rm tr})\cong\coh^*(A,M)$ is finitely generated over $\HH^*(A)$. The action of $\coh^*(A,k)$ factors through that of $\HH^*(A)$~\cite[Lemma 7.3]{PW} and $\HH^*(A)$ is a finitely generated module over $\coh^*(A,k)$. Hence $\coh^*(A,M)$ is finitely generated over $\coh^*(A,k)$. Proposition \ref{prop:1}(3) implies that $\coh^*(A,k)$ is noetherian. So $A$ satisfies {\rm \textbf{(hfg)}}. 
\end{proof}


\section{On the equivalency of finite generation conditions for Hopf algebras}
\label{sec:fg conditions}

\numberwithin{equation}{section}

In this section, we provide some equivalent descriptions of the finite generation conditions on the cohomology of any finite dimensional Hopf algebra. First of all, we assert that the assumption {\rm \textbf{(hfg)}} is preserved under any field extension: 

\begin{lemma}
\label{Rem:ext}
Let $A$ be a finite dimensional Hopf algebra over a field $k$, and let $K$ be any field extension of $k$. Then $A$ satisfies {\rm \textbf{(hfg)}} if and only if $A'=A\otimes_k K$ satisfies {\rm \textbf{(hfg)}}.
\end{lemma}

\begin{proof}
Suppose $A'$ satisfies {\rm \textbf{(hfg)}}. Let $M$ be a finite module over $A$, and $M'=M\otimes_kK$ the corresponding finite module over $A'$.  Since $K$ is flat over $k$, there is a graded ring isomorphism  $\coh^*(A',K)\cong \coh^*(A,k)\otimes_k K$. Hence there is an embedding of categories from $\coh^*(A,k)$-modules to $\coh^*(A',K)$-modules given by the tensor product $-\otimes_kK$. By hypothesis, $\coh^*(A',K)$ is noetherian, and it now follows that $\coh^*(A,k)$ is noetherian. 

It remains to show that $\coh^*(A,M)$ is finitely generated over $\coh^*(A,k)$. It suffices to show that $\coh^*(A,M)$ is finitely presented. By a result of Lenzing \cite[Satz 3]{Len}, it is equivalent to show that $\Hom_{\coh^*(A,k)}(\coh^*(A,M),-)$ preserves any inductive limit $\varinjlim M_i$ in the category of $\coh^*(A,k)$-modules. There is a natural map
\[
\xymatrix{
\varinjlim \Hom_{\coh^*(A,k)}\left(\coh^*(A,M),\,M_i\right)\ar[rr]^-{f}&& \Hom_{\coh^*(A,k)}\left(\coh^*(A,M),\,\varinjlim M_i\right).
}
\]
After applying $-\otimes_kK$, it becomes 
\[
\xymatrix{
\varinjlim \Hom_{\coh^*(A',K)}\left(\coh^*(A',M'),\,M_i'\right)\ar[rr]^-{f\otimes_kK}&& \Hom_{\coh^*(A',K)}\left(\coh^*(A',\,M'),\,\varinjlim M_i'\right),
}
\]
where $M_i'=M_i\otimes_kK$. The map $f\ot_kK$ is an isomorphism since $\coh^*(A',M')$ is finitely generated over the noetherian algebra $\coh^*(A',K)$ and hence it is finitely presented and preserves the inductive limit $\varinjlim M_i'$ in the category of $\coh^*(A',K)$-modules. This implies that $f$ is an isomorphism, since $-\otimes_kK$ is exact. 

On the other hand, suppose $A$ satisfies {\rm \textbf{(hfg)}}. We first deal with the case when $K/k$ is a finite field extension. By hypothesis and by the graded commutativity of $\coh^*(A,k)$, it is easy to see that $\coh^*(A',K)\cong\coh^*(A,k)\otimes_kK$ is noetherian and $K$-affine. Let $M'$ be a finite module over $A'$, and $M=\!_AM'$ be the restriction of $M'$ to $A$, which is again a finite $A$-module since $A'=A\otimes_kK$ is a free $A$-module of rank $[K:k]$. Then $\coh^*(A,M)$ is finitely generated over $\coh^*(A,k)$ since $A$ satisfies {\rm \textbf{(hfg)}}. Moreover, by the tensor-hom adjoint pair, there are isomorphisms of graded vector spaces: 
\begin{align*} \coh^*(A,M) &~\cong~\Ext^*_{A}(k,\Hom_{A'}(\!_{A'}{A'}_{A},M')) \\
&~\cong~\Ext^*_{A'}(A'\otimes_{A}k,M') \\
&~\cong~\Ext^*_{A'}(K,M')~=~\coh^*(A',M').
\end{align*} 
Thus, $\coh^*(A',M')$ is finitely generated over $\coh^*(A',K)=\coh^*(A,k)\otimes_kK$. This implies that $A'$ satisfies {\rm \textbf{(hfg)}} whenever $K/k$ is a finite field extension. 

In general, say $k\subset K$ is any field extension. By the previous discussion, it suffices to show that $A'=A\otimes_k\overline{K}$ satisfies {\rm \textbf{(hfg)}}, where $\overline{K}$ is the algebraic closure of $K$. We can argue similarly that $\coh^*(A', \overline{K})=\coh^*(A,k)\otimes_k\overline{K}$ is noetherian and finitely generated. Since $A$ is finite dimensional, there exists a finite field extension $F$ of $k$ such that the quotient $A\otimes_kF/\text{rad}(A\otimes_kF)$ is a direct sum of matrix algebras over $F$ (for instance, we can take $F$ to be the splitting field of $A/\text{rad}(A)$). Now let $B=A\otimes_kF$, so $B\otimes_F\overline{K}=A'$ and $(B/\text{rad}(B))\otimes_F\overline{K}\cong A'/\text{rad}(A')$. Since $F/k$ is a finite field extension, $B$ satisfies {\rm \textbf{(hfg)}} by the previous discussion. Therefore, $\coh^*(B,B/\text{rad}(B))$ is finitely generated over $\coh^*(B,F)$ and 
$$
\coh^*(B,B/\text{rad}(B))\otimes_F\overline{K}~\cong~\coh^*(B\otimes_F\overline{K},(B/\text{rad}(B))\otimes_F\overline{K})~\cong~\coh^*(A',A'/\text{rad}(A'))
$$
is finitely generated over $\coh^*(A',\overline{K})=\coh^*(B,F)\otimes_F\overline{K}$. This implies that $\coh^*(A',S)$ is finitely generated over $\coh^*(A',\overline{K})$ for any $A'$-simple  $S$. Finally, by filtering any finite dimensional $A'$-module by its composition series, we see that $A'$ satisfies {\rm \textbf{(hfg)}}.
\end{proof}

\begin{remark}
\label{rem:FE}
Using a similar argument, we can show that for any finite dimensional (not necessarily Hopf) algebra $A$ over a field $k$, and any field extension $K$ of $k$, $A$ satisfies the Hochschild condition {\rm \textbf{(fg)}} if and only if $A'=A\otimes_k K$ satisfies {\rm \textbf{(fg)}}.
\end{remark}

In the classical case of group cohomology, it is well-known that the group algebra $kG$ of any finite group $G$ satisfies \textbf{(hfg)}. More generally, Evens proved in \cite[Theorem 8.1]{Evens} that for any affine commutative algebra $R$ which admits a $kG$-module algebra structure, $\coh^*(G,M):= \coh^*(kG,M)$ is a finitely generated module over $\coh^*(G,R)$ for any finitely generated module $M$ over the skew group ring $R\#kG$. This nice property of finite group cohomology prompts us to seek an analogy in the reign of finite dimensional Hopf algebras. Thus we consider the following generalized noetherian assumption on a finite dimensional Hopf algebra $A$, where $R\# A$ denotes a smash product algebra (see~\cite{MO93}): 

\vspace{0.5em}
\begin{itemize}
\item[\textbf{(hfg*)}:] $\coh^*(A,M)$ is a noetherian module over $\coh^*(A,R)$ for any affine commutative $A$-module algebra $R$ and any finitely generated $R\#A$-module $M$.
\end{itemize}
\vspace{0.5em}

It is clear that \textbf{(hfg*)} implies \textbf{(hfg)} by taking $R=k$. The converse implication will also be true under an additional assumption, namely the integral property of $A$ (see~Section~\ref{subsec:integral}) as we show in the following lemma.

\begin{lemma}
\label{lem:hfg}
Let $A$ be a finite dimensional Hopf algebra over a field $k$. Suppose $A$ has the integral property. Then the following are equivalent.
\begin{enumerate} 
\item $A$ satisfies {\rm \textbf{(hfg)}}.
\item $A$ satisfies {\rm \textbf{(hfg*)}}. 
\item $\coh^*(A,R)$ is finitely generated for any affine commutative $A$-module algebra $R$.
\item Let $T=\bigoplus_{i\ge 0}T_i$ be a finitely generated graded noetherian $A$-module algebra that is a finite module over some graded central $A$-module subalgebra $Z$ of $T$, and let $M$ be a finitely generated module over $T\#A$. Then $\coh^*(A,M)$ is noetherian over $\coh^*(A,T)$. 
\end{enumerate}
\end{lemma}

\begin{proof}
(i)$\Rightarrow$(ii): Without loss of generality, by Lemma \ref{Rem:ext}, we can replace $k$ by a finite field extension and assume $A/{\rm rad}(A)$ is a direct sum of matrix algebras over $k$. Let $R$ be an affine commutative $A$-module algebra. 

We first treat the case when $A$ acts on $R$ trivially such that $R\# A\cong R\ot A$. Let $M$ be an $R\ot A$-module.  If $M$ is cyclic, we can write $M=(R\otimes A)/I$ for some left ideal $I$ of $R\otimes A$. Since $A$ is finite dimensional, there is a composition series $0=V_0\subset V_1\subset V_2\cdots \subset V_n=A$ of left $A$-modules, where each factor $S_{i}=V_{i}/V_{i-1}$ is a simple $A$-module. This induces a finite filtration $0=M_0\subset M_1\subset M_2\cdots \subset M_n=M$ on $M$ after applying $R\otimes_k-$, where each factor $M_{i}/M_{i-1}$ is a quotient module of $R\otimes S_{i}$. Note that $R\otimes S_{i}$ is a cyclic module over $R\otimes (A/{\rm Ann}(S_{i}))\cong R\otimes M_{d_{i}}(k)\cong M_{d_{i}}(R)$. By replacing $R\otimes A$ with $M_{d_i}(R)$, one sees that $R\otimes S_i\subset M_{d_i}(R)$ is just column matrices $R^{d_i}$. Then by some matrix multiplication, $M_i/M_{i-1}\cong (R/J_i)\otimes S_i$ for some ideal $J_i$ of $R$. Since $A$ acts on $R$ trivially, by the universal coefficients theorem, $\coh^*(A,R)\cong R\otimes \coh^*(A,k)$, and so is finitely generated and noetherian. Moreover, $\coh^*(A,(R/J_i)\otimes S_i)=(R/J_i)\otimes \coh^*(A,S_i)$ is finitely generated over $\coh^*(A,R)=R\otimes \coh^*(A,k)$ since $\coh^*(A,S_i)$ is finitely generated over $\coh^*(A,k)$ by {\rm \textbf{(hfg)}}. For each $i$, by applying $\coh^*(A,-)$ to the short exact sequence $0\to M_{i-1}\to M_i\to (R/J_i)\otimes S_i\to 0$, we obtain an exact sequence
\[
\xymatrix{
\coh^*(A,M_{i-1})\ar[r]& \coh^*(A,M_i)\ar[r]& \coh^*(A,(R/J_i)\otimes S_i),
}
\]
where we have just shown that $\coh^*(A,(R/J_i)\otimes S_i)$ is finitely generated over $\coh^*(A,R)$. Then induction on $i$ yields that $\coh^*(A,M_i)$ is finitely generated for all $i$. This completes the cyclic case.  In general, we can again induct on the number of minimal generators of $M$ and employ an exact sequence similar to that above to conclude this trivial $A$-action case. 

Finally, if $A$ acts on $R$ arbitrarily, by the integral property of $A$, $R$ is a finitely generated module over the invariant subring $R^A$. Then by the previous discussion, $\coh^*(A,M)$ is finitely generated over $\coh^*(A,R^A)$ and hence is finitely generated over $\coh^*(A,R)$. In particular, when $M=R$, we conclude that $\coh^*(A,R)$ is finitely generated as a module over $\coh^*(A,R^A)=R^A\otimes \coh^*(A,k)$, which is noetherian.  

(iii)$\Rightarrow$(i): By letting $R=k$, we know $\coh^*(A,k)$ is finitely generated and noetherian. For any finite $A$-module $M$, denote $R=k\bigoplus M$ where $M^2=0$. Then one can check that $\coh^*(A,R)$ is finitely generated implies that $\coh^*(A,M)$ is finitely generated over $\coh^*(A,k)$. 

(ii)$\Rightarrow$(iv): Since $A$ satisfies {\rm \textbf{(hfg*)}}, $\coh^*(A,M)$ is noetherian over $\coh^*(A,Z^{{\rm ev}})$. As the action of $\coh^*(A,Z^{{\rm ev}})$ on $\coh^*(A,M)$ factors through that of $\coh^*(A,T)$, it follows that $\coh^*(A,M)$ is noetherian over $\coh^*(A,T)$.

(ii)$\Rightarrow$(iii) and (iv)$\Rightarrow$(ii) are clear. 
\end{proof}

Now we are able to summarize various equivalent finite generation conditions on the cohomology of a finite dimensional Hopf algebra. 

\begin{prop}
\label{equivfg*}
For a finite dimensional Hopf algebra $A$ over a field $k$, the following are equivalent: 
\begin{enumerate}
\item $A$ satisfies {\rm \textbf{(hfg)}}.
\item $A$ satisfies {\rm \textbf{(fg)}} and $\HH^*(A)$ is a finitely generated module over $\coh^*(A,k)$.
\item $\coh^*(A,k)$ is a finitely generated algebra and $\Ext^*_A(A/\rm{rad}(A),A/\rm{rad}(A))$ is a finitely generated module over $\coh^*(A,k)$.  
\item $\coh^*(A,k)$ is a finitely generated algebra and $\Ext^*_A(k,A/\rm{rad}(A))$ is a finitely generated module over $\coh^*(A,k)$.  
\item $\coh^*(A,k)$ is a finitely generated algebra and $\Ext^*_A(M,N)$ is a finitely generated module over $\coh^*(A,k)$ for all pairs of finite $A$-modules $M$ and $N$.
\item $\coh^*(A,k)$ is a finitely generated algebra and $\Ext^*_A(k,M)$ is a finitely generated module over $\coh^*(A,k)$ for any finite $A$-module $M$.
\end{enumerate}
If in addition $A$ has the integral property, each of the above conditions is equivalent to {\rm \textbf{(hfg*)}}. 
\end{prop}

\begin{proof}
(i)$\Leftrightarrow$(ii): is Proposition \ref{adjoint}.

(ii)$\Leftrightarrow$(iii)$\Leftrightarrow$(v): These implications follow directly from \cite[Proposition 1.4]{EHSS} (see Proposition \ref{equivfg}), where we take $\coh^*(A,k)$ as a subalgebra of $\HH^*(A) \cong \coh^*(A,A^{\rm ad})=\coh^*(A,k) \oplus \coh^*(A,I)$ \cite[Lemma 7.2]{PW}. Here $A^{\rm ad}$ is the adjoint representation of $A$ and $I$ is the augmentation ideal of $A$. Note that in each case, $\HH^*(A)$ is a finitely generated module over $\coh^*(A,k)$ and hence it is noetherian.

(i)$\Leftrightarrow$(iv): Direction ``$\Rightarrow$" is clear and the other direction comes from filtering a finite $A$-module by its composition series and noting that (iv) implies $\coh^*(A,S)$ is finite generated over $\coh^*(A,k)$ for any simple $A$-module $S$.

(i)$\Leftrightarrow$(vi): It is clear since $\coh^*(A,k)$ is graded commutative. 

Finally, the last statement is a consequence of Lemma~\ref{lem:hfg}.
\end{proof}

\begin{remarks}
\label{rem:hfg}\  
\begin{itemize}
\item[(a)] It is straightforward to check that the condition {\rm \textbf{(hfg)}} is equivalent to the original finite generation condition in \cite{FW2015} which in characteristic $\neq 2$ says that $\coh^{\rm ev}(A,k)$ is finitely generated, and that for any pair of finite $A$-modules $M$ and $N$, $\Ext_A^*(M, N)$ is finitely generated over $\coh^{\rm ev}(A, k)$.
\item[(b)] More generally, {\rm \textbf{(hfg)}} can be stated for any augmented algebra $A$, and one can show that (i)$\Leftrightarrow$(iv)$\Leftrightarrow$(vi) in Proposition \ref{equivfg*} by further requiring $\coh^*(A,k)$ to be noetherian in (vi).
\end{itemize}
\end{remarks}

The equivalent finite generation conditions in Proposition~\ref{equivfg*} allow us to study the finite generation of (Hochschild) cohomology from various perspectives.


\section{A spectral sequence argument for the finite generation conditions}
\label{sec:homPI}

In this section, we use the Hilton-Eckmann argument in a suspended monoidal category to prove that the cohomology ring of $A$ with coefficients in an $A$-module algebra $R$ is finite over its affine graded center under some assumptions of $R$. As applications, we show the finite generation conditions hold for certain filtered, smash and crossed product algebras using the corresponding May and Lyndon-Hochschild-Serre spectral sequences. 

We start with a special case when the smash product is taken with a finite dimensional cocommutative Hopf algebra. We follow the argument first used in \cite[Corollary 3.2.2]{Benson1} for the key step to show that a particular subalgebra lies in the graded center.  

\begin{lemma}
\label{lem:cocom}
Let $A$ be a finite dimensional cocommutative Hopf algebra satisfying {\rm \textbf{(hfg)}} and let $R=\bigoplus_{i\ge 0} R_i$ be a graded $A$-module algebra. If $R$ is a finitely generated noetherian algebra that is a finite module over a graded central $A$-module subalgebra $Z$, then $\coh^*(A,R)=\bigoplus_{i+j\ge 0}\coh^i(A,R_j)$ is a finitely generated noetherian algebra and is a finite module over its graded center (grading by total degree).  
\end{lemma}

\begin{proof}
By Proposition \ref{prop:1}(2) and (3), $Z$ and $Z^{\rm ev}$ are finitely generated and noetherian. By Proposition~\ref{prop:integral}(4), $Z^{\rm ev}$ is a finitely generated module over its invariant subring $(Z^{\rm ev})^A$ and hence so is $R$. Write $W=(Z^{\rm ev})^A=\bigoplus_{i\ge 0} W_i$. Note that $\coh^*(A,W)\cong \coh^*(A,k)\otimes W$ is finitely generated noetherian. By Lemma \ref{lem:hfg}, $\coh^*(A,R)$ is a finitely generated module over $\coh^*(A,W)$. It remains to show that $\coh^*(A,W)$ has image in the graded center of $\coh^*(A,R)$, which follows from the commutative diagram:  
\[
\begin{xy}*!C\xybox{
\xymatrixcolsep{2pc}
\xymatrix{
 & P_m \ot P_n \ar[rrr]^{(-1)^{\alpha n} \, f_{m,\alpha} \ot g_{n,\beta}} \ar[dd]^{(-1)^{mn} \, \tau} &&&W_{\alpha} \ot R_{\beta} \ar[rrr]^{m_R} \ar[dd]^{(-1)^{mn+n\alpha+\beta m} \, \tau} &&& R_{\alpha+\beta} \\
 P_{\bu} \ar[ur]^{\Delta} \ar[dr]_{\Delta'} &&&&&&& \\
 & P_n \ot P_m \ar[rrr]^{(-1)^{\beta m} \, g_{n,\beta} \ot f_{m,\alpha}} &&&  R_{\beta} \ot W_{\alpha} \ar[rrr]^{(-1)^{(m+\alpha)(n+\beta)}} &&& R_{\beta} \ot W_{\alpha} \ar[uu]_{m_R},
}}
\end{xy}
\]
where $P_{\bu}$ is a projective resolution of $k$ over $A$, $\Delta: P_{\bu} \rightarrow P_{\bu} \ot P_{\bu}$ is a diagonal map, $\Delta '$ is another diagonal map (defined by commutativity of the left triangle), 
$f_{m,\alpha} \in \coh^m(A, W_\alpha)$, $g_{n,\beta} \in \coh^n(A,W_\beta)$, $\tau$ is the twisting map (which is always a morphism of $A$-modules since $A$ is cocommutative), and $m_R$ is the multiplication in $R$. The signs $(-1)^{\alpha n}$ and $(-1)^{\beta m}$ on the rows of the first square come from a standard sign convention. The second square demonstrates the fact that $W$ lies in the graded center of $R$. It follows from the commutativity of the diagram that 
$$f_{m,\alpha} \smile g_{n,\beta} = (-1)^{(m+\alpha)(n+\beta)} \, g_{n,\beta} \smile f_{m,\alpha}.$$
\end{proof}

In general, in order to obtain needed results when $A$ is an arbitrary finite dimensional Hopf algebra, we follow the Hilton-Eckmann argument in the context of a suspended monoidal category demonstrated in \cite{Suarez-Alvarez}. A good reference for all the terminologies is \cite{EGNO}.

\begin{definition}
A {\bf suspended monoidal category} is a 9-tuple $(\mathcal C,\otimes,e,a,\ell,r,T,\lambda,\rho)$ such that $(\mathcal C,\otimes,e,a,r,\ell)$ is a monoidal category, $T:\mathcal C\to \mathcal C$ is an automorphism, $\lambda_{X,Y}: X\otimes TY\to T(X\otimes Y)$ and $\rho:TX\otimes Y\to T(X\otimes Y)$ are isomorphisms of functors $\mathcal C\times \mathcal C\to \mathcal C$ for each pair of objects $X$ and $Y\in {\rm obj}\,\mathcal C$, and the following diagrams commute:
\[
\xymatrix{
e\otimes TX\ar[r]^-{\ell}\ar[d]_-{\lambda} & TX\ar[d]^-{1}\\
T(e\otimes X)\ar[r]^-{T\ell} & TX
}
\qquad \qquad \qquad 
\xymatrix{
TX\otimes e\ar[r]^-{r}\ar[d]_-{\rho} & TX\ar[d]^-{1}\\
T(X\otimes e)\ar[r]^-{Tr} & TX,
}
\]
while the following diagram anti-commutes: 
\[
\xymatrix{
TX\otimes TY\ar[r]^-{\rho}\ar[d]_-{\lambda}\ar@{}[dr]|{(-1)}&T(X\otimes TY)\ar[d]^-{T\lambda}\\
T(TX\otimes Y)\ar[r]^-{T\rho}& T^2(X\otimes Y).
}
\]
\end{definition}
Given a suspended monoidal category, as in \cite[\S 1.5 and \S1.6]{Suarez-Alvarez}, we can inductively define a series of isomorphisms of functors $\mathcal C\times \mathcal C\to \mathcal C$ by  
$$\lambda_q: X\otimes T^qY\to T^q(X\otimes Y) , \quad \rho_p: T^pX\otimes Y\to T^p(X\otimes Y),$$
for all $p,q\in \mathbb Z$ with $\lambda_1=\lambda, \rho_1=\rho$, such that the following diagrams commute:  
\[
\xymatrix{
e\otimes T^qX\ar[r]^-{\ell}\ar[d]_-{\lambda_q} & T^qX\ar[d]^-{1}\\
T^q(e\otimes X)\ar[r]^-{T^q \ell} & T^qX
}
\qquad \qquad \qquad 
\xymatrix{
T^pX\otimes e\ar[r]^-{r}\ar[d]_-{\rho_p} & T^pX\ar[d]^-{1}\\
T^p(X\otimes e)\ar[r]^-{T^pr} & T^pX,
}
\]
while the following diagram $(-1)^{pq}$-commutes: 
\[
\xymatrix{
T^pX\otimes T^qY\ar[r]^-{\rho_p}\ar[d]_-{\lambda_q}\ar@{}[dr]|{(-1)^{pq}}&T^p(X\otimes T^qY)\ar[d]^-{T^p\lambda_q}\\
T^q(T^pX\otimes Y)\ar[r]^-{T^q\rho_p}& T^{p+q}(X\otimes Y).
}
\]

Next, we have isomorphisms $\sigma: X\otimes e\cong e\otimes X$ and $\tau: e\otimes X\cong X\otimes e$ satisfying $\ell\sigma=r$ and $r\tau=\ell$ for any $X\in {\rm obj}\,\mathcal C$. Moreover, all the isomorphisms above can be naturally extended from $e$ to $Y$ whenever $Y=\bigoplus e$ is a direct sum of copies of the identity object $e$.

Now, let $R=\bigoplus_{i \in \mathbb Z} R_i$ be a graded ring in $\mathcal C $ with product maps $m_{ij}: R_i\otimes R_j\to R_{i+j}$ satisfying the associativity axiom: 
\[
\xymatrix{
(R_i \otimes R_j)\otimes R_k \ar[r]^-{m_{ij}\otimes 1}\ar[d]_-{a_{ijk}} & R_{i+j}\otimes R_k\ar[dd]^-{m_{(i+j)k}}\\
R_i\otimes (R_j\otimes R_k)\ar[d]_-{1\otimes m_{jk}} &\\
R_i\otimes R_{j+k}\ar[r]^-{m_{i(j+k)}} \ar[r] & R_{i+j+k}.
}
\]
Moreover, we say $R$ is {\bf unital} if there is a morphism $u: e\to R_0$ satisfying the unit axiom
\[
\xymatrix{
e\otimes R_i\ar[rr]^-{u\otimes 1}\ar[rd]_-{\ell} &&R_0\otimes R_i\ar[dl]^-{m_{0i}}\\
&R_i&
}\quad\quad\quad
\xymatrix{
R_i\otimes e\ar[rr]^-{1\otimes u}\ar[rd]_-{r} &&R_i\otimes R_0\ar[dl]^-{m_{i0}}\\
&R_i&
}.
\]
Consider the ring
$$
   E(R):=\bigoplus_{i,j \in \mathbb Z}\, \Hom_\mathcal C(e,T^iR_j), 
$$
where the product in $E(R)$ is given by the following composition for any $f: e\to T^pR_\alpha$ and $g: e\to T^qR_\beta$:  
\begin{equation}\label{eq:prod}
\small
\xymatrix{
f\cdot g:e  && e\otimes e\ar[ll]_-{\ell=r}\ar[rr]^-{(-1)^{q\alpha}f\otimes g} && T^pR_\alpha \otimes T^qR_\beta \ar[rr]^-{\rho_p}&& T^p(R_\alpha\otimes T^qR_\beta)\\
\ar[rr]^-{T^p\lambda_q}&&T^{p+q}(R_\alpha\otimes R_\beta)\ar[rr]^-{T^{p+q}m_{\alpha\beta}}&&T^{p+q}R_{\alpha+\beta}.
}
\end{equation}
The sign $(-1)^{q\alpha}$ above comes from the sign convention when passing $g$ over $R_{\alpha}$. Notice that if $R$ is unital, then $E(e)$ has image as a graded subalgebra of $E(R)$. Moreover, we say that $M=\bigoplus_{i\in \mathbb Z} M_i$ is a graded (bi)module over $R$ if there are morphisms $\ell_{ij}: R_i\otimes M_j\to M_{i+j}$ (resp. $r_{ji}: M_j\otimes R_i\to M_{i+j}$) satisfying some obvious compatibility conditions. Suppose $M=\bigoplus_{i\in \mathbb Z} M_i$ is a graded bimodule over $R=\bigoplus_{i \in \mathbb Z} R_i$ with each $R_i$ being a direct sum of copies of $e$. We call $M$ a \textbf{graded symmetric bimodule} over $R$ if the following diagrams commute:
\[
\xymatrix{
R_i\otimes M_j\ar[rr]^-{\tau}\ar[dr]_-{r_{ij}} &\ar@{}[d]|{(-1)^{ij}}& M_j\otimes R_i\ar[dl]^-{\ell_{ji}}\\
& M_{i+j}&
}\qquad\quad\qquad
\xymatrix{
M_j\otimes R_i\ar[rr]^-{\sigma}\ar[dr]_-{r_{ji}} &\ar@{}[d]|{(-1)^{ij}}& R_i\otimes M_j\ar[dl]^-{\ell_{ij}}\\
& M_{i+j}&
}
\]
Similarly we denote 
$$
   E(M):=\bigoplus_{i,j \in \mathbb Z}\, \Hom_\mathcal C(e,T^iM_j).
$$

The following theorem extends Su\'{a}rez-\'{A}lvarez's results \cite{Suarez-Alvarez} on the graded commutativity of the (Hochschild) cohomology ring of a finite dimensional (Hopf) algebra.

\begin{theorem}
\label{thm:gc}
Retain the above notations. Let $M=\bigoplus_{i\in \mathbb Z} M_i$ be a graded symmetric bimodule over $R=\bigoplus_{i \in \mathbb Z} R_i$ with each $R_i$ being a direct sum of copies of $e$. Then $E(M)$ is a graded symmetric bimodule over $E(R)$. In particular, $E(M)$ is a graded symmetric bimodule over $E(e)$. 
\end{theorem}

\begin{proof}
Take $f: e\to T^pM_\alpha$ and $g:e\to T^qR_\beta$. It is straightforward to check that we have the following $(-1)^{p\beta}$-commutative diagram:
\[\small
\xymatrix{
e\ar[d]_-{g}&  e\otimes e\ar[l]_-{r}\ar[r]^-{(-1)^{p\beta}g\otimes f}\ar[d]_-{g\otimes 1}\ar@{}[dr]|{(-1)^{p\beta}} &   T^qR_\beta \otimes T^pM_\alpha\ar[d]^-{1\otimes 1} \\
T^qR_\beta\ar[d]_-{1}  &T^qR_\beta \otimes e\ar[l]_-{r}\ar[r]^-{1\otimes f}\ar[d]_-{\rho_q}&T^qR_\beta \otimes T^pM_\alpha\ar[d]^-{\rho_q}\\
T^qR_\beta   &T^q(R_\beta \otimes e)\ar[l]_-{T^qr}\ar[r]^-{T^q(1\otimes f)}  &T^q(R_\beta\otimes T^pM_\alpha)\ar[r]^-{T^q\lambda_p}& T^{p+q}(R_\beta\otimes M_\alpha)\ar[r]^-{T^{p+q} \ell_{\beta\alpha}} & T^{p+q}M_{\alpha+\beta}
}
\]
where the right outer boundary represents $g\cdot f$. On the other hand, the right outer boundary of the $(-1)^{q\alpha+pq}$-commutative diagram below is equal to $f\cdot g$:
\[\small
\xymatrix{
e\ar[d]_-{g}&  e\otimes e\ar[l]_-{\ell=r}\ar[r]^-{(-1)^{q\alpha}f\otimes g}\ar[d]_-{1\otimes g}\ar@{}[dr]|{(-1)^{q\alpha}} &   T^pM_\alpha \otimes T^qR_\beta\ar[d]^-{1\otimes 1} \\
T^qR_\beta\ar[d]_-{1}  &e \otimes T^qR_\beta \ar[l]_-{\ell}\ar[r]^-{f\otimes 1}\ar[d]_-{\lambda_q}&T^pM_\alpha\otimes T^qR_\beta \ar@{}[dr]|{(-1)^{pq}}\ar[d]^-{\lambda_q}\ar[r]^-{\rho_p}& T^p(M_\alpha\otimes T^qR_\beta)\ar[d]^-{T^p\lambda_q}\\
T^qR_\beta   &T^q( e\otimes R_\beta)\ar[l]_-{T^q \ell}\ar[r]^-{T^q(f\otimes 1)}  &T^q(T^pM_\alpha\otimes R_\beta)\ar[r]^-{T^q\rho_p}& T^{p+q}(M_\alpha\otimes R_\beta)\ar[r]^-{T^{p+q}r_{\alpha\beta}} & T^{p+q}M_{\alpha+\beta}
}
\]
Therefore, the anti-commutativity of the product of $f$ and $g$ can be derived from the next $(-1)^{\alpha\beta}$-commutative diagram: 
\[\small
\xymatrix{
e\ar[d]_-{1}  &T^qR_\beta\ar[l]_{g}\ar[d]_-{1} &   T^q(e\otimes R_\beta)\ar[l]_{T^q \ell}\ar[r]^-{T^q(f\otimes 1)}\ar[d]_-{T^q\sigma} &    T^q(T^pM_\alpha\otimes R_\beta)\ar[r]^-{T^q\rho_p}\ar[d]_-{T^q\sigma}&  T^{p+q}(M_\alpha\otimes R_\beta)\ar[r]^-{T^{p+q}r_{\alpha\beta}} \ar[d]_-{T^{p+q}\sigma}\ar@{}[dr]|{(-1)^{\alpha\beta}}&T^{p+q}M_{\alpha+\beta}\ar[d]_-{1}\\ 
e & T^qR_\beta\ar[l]_{g}   & T^q(R_\beta \otimes e)\ar[l]_{T^qr}\ar[r]^-{T^q(1\otimes f)}   &   T^q(R_\beta \otimes T^pM_\alpha)\ar[r]^-{T^q\lambda_p} &T^{p+q}(R_\beta \otimes M_\alpha)\ar[r]^-{T^{p+q} \ell_{\beta\alpha}} & T^{p+q}M_{\alpha+\beta}, \\ 
}
\]
where the top row is equal to $(-1)^{q\alpha+pq}f\cdot g$ and the bottom row equals $(-1)^{p\beta}g\cdot f$. Hence $f\cdot g=(-1)^{(p+\alpha)(q+\beta)}g\cdot f$. 
Note that the third square commutes because of the functoriality of $\sigma$ and the fourth square commutes by commutativity of  
\[\small
\xymatrix{
T^pX\otimes e\ar@/_2.0pc/[dd]_-{\sigma}\ar[d]^-{r} \ar[r]^-{\rho_p}&   T^p(X\otimes e)\ar@/^2.0pc/[dd]^-{T^p\sigma}\ar[d]_-{T^pr}\\
T^pX\ar[r]^-{1} & T^pX\\
e\otimes T^pX\ar[r]^-{\lambda_p}\ar[u]_-{\ell}& T^p(e\otimes X)\ar[u]^-{T^p \ell}
}
\]
for any object $X\in {\rm obj}\, \mathcal C$. 
\end{proof}

\begin{corollary}
\label{cor:algebragc}
Let $(\mathcal C,\otimes,e,a,\ell,r,T,\lambda,\rho)$ be a suspended monoidal category. Let $W=\bigoplus_{i\in \mathbb Z} W_i$ be a subring of $R=\bigoplus_{i\in \mathbb Z} R_i$ with each $W_i$ being a direct sum of copies of $e$ and the following $(-1)^{ij}$-commutative diagrams 
\[
\xymatrix{
W_i\otimes R_j\ar[rr]^-{\tau}\ar[dr]_-{\mu} &\ar@{}[d]|{(-1)^{ij}}& R_j\otimes W_i\ar[dl]^-{\mu}\\
& R_{i+j}&
}\quad\quad\quad
\xymatrix{
R_i\otimes W_j\ar[rr]^-{\sigma}\ar[dr]_-{\mu} &\ar@{}[d]|{(-1)^{ij}}& W_j\otimes R_i\ar[dl]^-{\mu}\\
& R_{i+j}&
}
\]
Then the subring $E(W)$ has image in the graded center of $E(R)$.
\end{corollary}

\begin{corollary}\cite[Theorem 1.7]{Suarez-Alvarez}
\label{cor:gc}
Let $(\mathcal C,\otimes,e,a,\ell,r,T,\lambda,\rho)$ be a suspended monoidal category. Set 
$$
   E(e):=\bigoplus_{i \in \mathbb Z}\,\Hom_\mathcal C(e,T^ie).
$$
If $f: e\to T^pe$ and $g: e\to T^qe$, define $f\cdot g=T^qf\circ g: e\to T^{p+q}e$. Then $E(e)$ is a commutative graded ring. 
\end{corollary}

\begin{proof}
Let $R=M=e$ in Theorem~\ref{thm:gc}. Note that the product $f\cdot g=T^qf\circ g: e\to T^{p+q}e$ defined in the statement coincides with the formula \eqref{eq:prod}; see the proof of \cite[Theorem 1.7]{Suarez-Alvarez}. \\
\end{proof}

\noindent
\underline{{\bf Applications.}} For finite dimensional (Hopf) algebras, spectral sequences are powerful tools for handling their cohomology rings. Therefore, we are interested in multiplicative spectral sequences satisfying similar finite generation conditions. The spectral sequences we have in mind are May spectral sequences related to filtered algebras and Lyndon-Hochschild-Serre spectral sequences related to smash and crossed products. They are described explicitly in the Appendix for completeness. We use these spectral sequences to conclude the finite generation conditions for the original (Hopf) algebras if the initial pages of these spectral sequences satisfy the finite generation conditions and the corresponding spectral sequences collapse at certain pages in positive characteristic.

\begin{prop}
\label{prop:hfg}
Let $A$ be a finite dimensional Hopf algebra over a field $k$, $M$ be any finite dimensional $A$-module, and $R=\bigoplus_{i\ge 0} R_i$ be a connected graded $A$-module algebra. Then
\begin{enumerate}
\item $\coh^*(A,k)$ maps to the graded center of $\Ext_A^*(M,M)$. Moreover, if $A$ satisfies {\rm \textbf{(hfg)}}, then $\Ext_A^*(M,M)$ is noetherian and a finite module over its graded center.
\item $\coh^*(A,k)$ maps to the graded center of $\coh^*(A,R)=\bigoplus_{i,j\ge 0}\coh^i(A,R_j)$ (with respect to the total degree). Moreover, if $A$ satisfies {\rm \textbf{(hfg*)}} and $R$ is a finitely generated noetherian algebra and a finite module over some graded central $A$-module subalgebra, then $\coh^*(A,R)$ is noetherian and a finite module over its graded center. 
\end{enumerate}
\end{prop}

\begin{proof}
We use Theorem~\ref{thm:gc} where $\mathcal C$ is the left derived category of ${\rm Mod}(A)$ with $\otimes=\otimes_k$, $e=k$ and $T=[1]$ is the shift functor. 

(i) Let $V^*$ be the left dual of $V$ (see e.g., \cite[\S 2.10]{EGNO}). Note that $\Ext_A^*(V, V)\cong \Ext_A^*(k, V\otimes V^*)$ and the standard actions of $\coh^*(A, k)$ on either one correspond under this isomorphism. By Theorem~\ref{thm:gc}, $\coh^*(A,k)$ maps to the graded center of $\Ext_A^*(k, V\otimes V^*)$ via the coevaluation map $\text{coev}: k\to V\otimes V^*$. Moreover, if $A$ satisfies {\rm \textbf{(hfg)}}, then $\Ext_A^*(k, V\otimes V^*)$ is finitely generated as a module over the noetherian graded central subalgebra given by the image of $\coh^*(A,k)$, and as a consequence it is noetherian. 

(ii) Since $k=R_0$ certainly lies in the graded center of $R$, $\coh^*(A,k)$ maps to the graded center of $\coh^*(A,R)$. Now assume $H$ satisfies {\rm \textbf{(hfg*)}} and denote again by $Z$ a graded central $A$-module subalgebra of $R$. Write $W=(Z^{\rm ev})^A=\bigoplus_{i\ge 0} W_i$, where $R$ is finitely generated as a module over $W$. Then $\coh^*(A,R)$ is a finitely generated module over the noetherian algebra $\coh^*(A,W)$ by {\rm \textbf{(hfg*)}} and hence itself is finitely generated noetherian.  We use Corollary \ref{cor:algebragc} to conclude that $\coh^*(A,W)$ maps to the graded center of $\coh^*(A,R)$ since each homogenous component $W_i$ of $W$ is a direct sum of copies of the trivial module $k$. 
\end{proof}

Similarly, we obtain the following result, where part (i) is a special case of \cite[Theorem 1.1]{SO}.

\begin{prop}
\label{prop:fg}
Let $A$ be a finite dimensional algebra over a field $k$, $M$ be a finite dimensional $A$-module, and $R$ be a finite dimensional unital algebra in ${\rm Mod}(A^e)$. Then 
\begin{enumerate}
\item  $\HH^*(A)$ maps to the graded center of $\Ext_A^*(M,M)$. Moreover, if $A$ satisfies {\rm \textbf{(fg)}}, then $\Ext_A^*(M,M)$ is noetherian and a finite module over its graded center.
\item  $\HH^*(A)$ maps to the graded center of $\HH^*(A,R)$. Moreover, if $A$ satisfies {\rm \textbf{(fg)}}, then $\HH^*(A,R)$ is noetherian and a finite module over its graded center.
\end{enumerate}
\end{prop}

\begin{proof}
Again, apply Theorem~\ref{thm:gc} where $\mathcal C$ is the left derived category of ${\rm Mod}(A^e)$ with $\otimes=\otimes_A$, $e=A$ and $T=[1]$ is the shift functor. 

(i) Note that $\Ext_A^*(M, M)\cong \HH^*(A, \Hom_k(M,M))$ and the standard actions of $\HH^*(A)$ on either one correspond under this isomorphism. By Theorem~\ref{thm:gc}, $\HH^*(A)$ maps to the graded center of $\HH^*(A, \Hom_k(M,M))$. 
Moreover, if $A$ satisfies {\rm \textbf{(fg)}}, then $\HH^*(A, \Hom_k(M,M))$ is finitely generated as a module over the noetherian graded central subalgebra $\HH^*(A)$, and as a consequence it is noetherian. 

(ii) Since $R$ is $A$-unital, by Corollary \ref{cor:algebragc}, $\HH^*(A)$ maps to the graded center of $\HH^*(A,R)$. Moreover, if $A$ satisfies {\rm \textbf{(fg)}}, $\HH^*(A,R)$ is finitely generated as a module over the noetherian graded central subalgebra $\HH^*(A)$, and as a consequence it is noetherian. 
\end{proof}

\begin{lemma} 
\label{lem:key}
Let $R$ be a commutative noetherian ring with characteristic $m>0$. Let $\{E^{i,j}_r\}_{r,i,j}$ be a convergent multiplicative spectral sequence of $R$-algebras concentrated in the half plane $i+j \ge 0$. Assume that for some $r_0\ge 1$,
\begin{enumerate}
\item $E_{r_0}^{*,*}$ is a finitely generated module over its graded center 
(grading by total degree), and 
\item $E_{r_0}^{*,*}$ is a noetherian $R$-algebra.
\item The spectral sequence $\{E^{i,j}_r\}_{r,i,j}$ collapses at some page $r_1\ge r_0$.
\end{enumerate}
Then $E_\infty^{*,*}$ is a noetherian $R$-algebra. 

Additionally, let $\{\widetilde{E}^{i,j}_r\}_{r,i,j}$ be a convergent spectral sequence that is a differential bigraded module over $\{E^{i,j}_r\}_{r,i,j}$. Suppose that for the same value $r_0$, $\widetilde{E}_{r_0}^{*,*}$ is finitely generated over $E_{r_0}^{*,*}$. Then $\widetilde{E}_\infty^{*,*}$ is finitely generated over $E_\infty^{*,*}$.
\end{lemma}

\begin{proof}
For each $r\geq 0$, let $C^{*,*}_r$ be the graded center of $E_r^{*,*}$.
Then for each $c\in C^{i,j}_r$ and $x\in E^{i',j'}_r$, 
\begin{equation}\label{eqn:key-cx}
   cx = (-1)^{(i+j)(i'+j')} x c 
\end{equation}
and $d(c)\in C^{i-1,j}_r \oplus C^{i, j-1}_r$.
As a consequence of equation~(\ref{eqn:key-cx}), $d(c)$ commutes with $c$, 
i.e.,~$c d(c) = d(c) c$. Thus if $i$ is odd, then
\[
  d(c^2)= d(c) c + (-1)^i c d(c) = 0 ,
\]
while if $i$ is even, since $m = \chara (R)$, we have
\[
   d(c^m) = d(c) c^{m-1} + c d(c) c^{m-2} +
   \cdots + c^{m-1} d(c) = m d(c) c^{m-1}=0 .
\]
It follows that for each $c$ in $C^{i,j}_r$ there is a finite positive power of $c$ that is
a cycle (i.e.,~$c^2$ if $i$ is odd, and $c^m$ if $i$ is even).  

By hypothesis, $E^{*,*}_{r_0}$ is a noetherian $R$-algebra
that is a finitely generated module over $C^{*,*}_{r_0}$.
By Proposition~\ref{prop:1}(3), $C^{*,*}_{r_0}$ is noetherian 
and finitely generated.
Let $c_1,\ldots, c_n$ be a set of homogeneous generators of
$C^{*,*}_{r_0}$. 
As above, for each $i$, some positive power of $c_i$ is a cycle. By repeating the process, for each $i$, there is a finite positive power, say $c_i^{t_i}$, that is a permanent cycle since the spectral sequence $\{E^{i,j}_r\}_{r,i,j}$ collapses at some page $r_1\ge r_0$. Let
$$A^{*,*}_{r_0}:= R \langle c_1^{t_1},\ldots, c_n^{t_n} \rangle \subseteq C^{*,*}_{r_0},$$
that is, $A_{r_0}^{*,*}$ is the subalgebra of $C_{r_0}^{*,*}$ generated by $c_1^{t_1},\ldots, c_n^{t_n}$.
By its definition, $A_{r_0}^{*,*}$ consists of permanent cycles and $C^{*,*}_{r_0}$ is
finitely generated as a module over $A^{*,*}_{r_0}$. It follows that $E^{*,*}_{r_0}$ is also a finitely generated module
over $A^{*,*}_{r_0}$.
For all $r>r_0$, let 
$$A_r^{*,*} := A^{*,*}_{r-1} / B_{r-1}^{*,*} \subseteq E^{*,*}_r ,$$ 
the subalgebra of $E^{*,*}_r$ given by
the quotient of $A^{*,*}_{r-1}$ by its ideal $B_{r-1}^{*,*}$ 
consisting of coboundaries. 
By construction, we have a sequence of algebras:
\[
\xymatrix{
  A^{*,*}_{r_0} \ar@{->>}[r] & \cdots \ar@{->>}[r]
   & A_r^{*,*} \ar@{->>}[r] & A _{r+1}^{*,*} \ar@{->>}[r]& \cdots ,
}
\]
for which each $A_i^{*,*}$ is a subalgebra of $E_i^{*,*}$. 
For each $r \geq r_0$, let $\Lambda ^{*,*}_r$ be the subalgebra
consisting of \emph{all permanent} cycles in $E^{*,*}_r$. 
A calculation shows that since $A^{*,*}_r$ consists of permanent cycles,  
$d_r (E^{*,*}_r)$ is an $A^{*,*}_r$-submodule of $\Lambda^{*,*}_r$.

Writing $\Lambda^{*,*}_{r+1} := \Lambda_r^{*,*}/ d_r( E_r^{*,*})$
for each $r$, we have a sequence of $A^{*,*}_{r_0}$-modules:
\[
\xymatrix{
  \Lambda^{*,*}_{r_0} \ar@{->>}[r] & \cdots \ar@{->>}[r]
   & \Lambda_r^{*,*} \ar@{->>}[r] & \Lambda _{r+1}^{*,*} \ar@{->>}[r]& \cdots .
}
\]
Now $\Lambda^{*,*}_{r_0}$ is an $A^{*,*}_{r_0}$-submodule of
$E^{*,*}_{r_0}$, and so is a noetherian $A^{*,*}_{r_0}$-module.
Let $K_r$ be the kernel of the surjection from $\Lambda^{*,*}_{r_0}$
to $\Lambda_r^{*,*}$. Then $K_r$ is also a noetherian $A^{*,*}_{r_0}$-module.
There is an increasing chain of submodules of $\Lambda^{*,*}_{r_0}$:
\[
   K_{r_0} \subseteq  K_{r_0+1} \subseteq \cdots .
\]
Since $\Lambda^{*,*}_{r_0}$ is noetherian, this chain stabilizes, 
that is $K_s = K_{s+1} = \cdots$, for some $s \ge r_0$.
Therefore, $\Lambda^{*,*}_s = E^{*,*}_{\infty}$, and
$E^{*,*}_{\infty}$ is itself a noetherian $A^{*,*}_{r_0}$-module. 

By~\cite[Proposition~2.1]{Evens}, $E^{*,*}_{\infty}$ is a noetherian module
over $\Tot (A^{*,*}_{r_0})$.
Since $A^{*,*}_{r_0}$ is finitely generated, it follows
that $E^{*,*}_{\infty}$ is a noetherian $R$-algebra.

For the remaining statement regarding $\widetilde{E}^{i,j}$, under the above set up and hypotheses, $\widetilde{E}^{*,*}_{r_0}$ is also a finitely generated module over $A^{*,*}_{r_0}$. For each $r \geq r_0$, let $\widetilde{\Lambda} ^{*,*}_r$ be the subalgebra consisting of \emph{all} permanent cycles in $\widetilde{E}^{*,*}_r$ and define $\widetilde{\Lambda}^{*,*}_{r+1} := \widetilde{\Lambda}_r^{*,*}/ d_r( \widetilde{E}_r^{*,*})$. Now $\widetilde{\Lambda}^{*,*}_{r_0}$ is an $A^{*,*}_{r_0}$-submodule of
$\widetilde{E}^{*,*}_{r_0}$, and so is a finitely generated $A^{*,*}_{r_0}$-module. Similar arguments show that $\widetilde{E}^{*,*}_{\infty}$ is a noetherian $A^{*,*}_{r_0}$-module and is finitely generated over $E^{*,*}_{\infty}$.
\end{proof}


\noindent
{\bf (Hochschild) Cohomology of filtered algebras:}
Let $A$ be a finite dimensional filtered (Hopf) algebra over a field $k$, and denote by $\gr A$ the corresponding associated graded (Hopf) algebra. Then there exists a May spectral sequence computing the (Hochschild) cohomology of $A$ in terms of the (Hochschild) cohomology of $\gr A$ as its first page. Similarly, there is a May spectral sequence that computes the (Hochschild) cohomology of $A$ with coefficients in any finite $A$-module $M$. Moreover, these spectral sequences inherit structures that are compatible with the cup product in the (Hochschild) cohomology of $A$ and its module structure on that of $M$ (see Appendix~\ref{subsec:May} and Appendix \ref{subsec:May-HH}).

\begin{theorem} 
\label{thm:FFGC}
Let $A$ be a finite dimensional filtered algebra (resp.~Hopf algebra) over a field $k$ of positive characteristic. If the associated graded algebra (resp.~Hopf algebra) $\gr A$ satisfies {\rm \textbf{(fg)}} (resp.~{\rm \textbf{(hfg)}}) and the May spectral sequence used to compute $\HH^*(A)$ (resp.~$\coh^*(A,k)$) collapses, then $A$ satisfies {\rm \textbf{(fg)}} (resp.~{\rm \textbf{(hfg)}}).
\end{theorem}
\begin{proof}
It follows from Lemma \ref{lem:key} and the discussion above since the (Hochschild) cohomology ring of $A$ is always graded commutative. 
\end{proof}


\noindent
{\bf Cohomology of crossed products:} Let $A=R \#_\sigma H$ be the crossed product of two finite dimensional Hopf algebras $R$ and $H$ over a field $k$ with respect to a cocycle $\sigma$ (see~\cite{MO93}). We assume that the augmentations of $R$ and $H$ are preserved under the crossed product so that $A$ is again augmented. In this case, there are Lyndon-Hochschild-Serre spectral sequences associated to the crossed product $A=R \#_\sigma H$ and any finite dimensional $A$-module $M$:
\begin{align*}
E_2^{p,q}(A)=\coh^p(H,\coh^q(R,k))&~\Longrightarrow~ \coh^{p+q}(A,k)=E_\infty^{p,q}(A), \\
E_2^{p,q}(M)=\coh^p(H,\coh^q(R,M))&~ \Longrightarrow~ \coh^{p+q}(A,M)=E_\infty^{p,q}(M).
\end{align*}
These spectral sequences inherit structures that are compatible with the multiplicative structure of $\coh^*(A,k)$ and its module structure on $\coh^*(A,M)$ (see Appendix~\ref{subsec:cohomology-smash}).

\begin{theorem}
\label{thm:FiniteTypeCoh}
Retain the notations above. Further assume that $R$ satisfies {\rm \textbf{(hfg)}} and $H$ satisfies {\rm \textbf{(hfg*)}}. Then $A=R \#_\sigma H$ satisfies {\rm \textbf{(hfg)}} if the corresponding LHS spectral sequence $E^{*,*}_r(A)$ collapses for 
\begin{enumerate}
\item $r=2$ when $\chara(k)=0$, or 
\item some $r\ge 2$ when $\chara(k)>0$.
\end{enumerate}
\end{theorem}
\begin{proof}
Since $R$ satisfies {\rm \textbf{(hfg)}}, the cohomology ring $\coh^*(R,k)$ is graded commutative and finitely generated. In view of Proposition~\ref{prop:dgmod-ss-coh} and the fact that $H$ satisfies {\rm \textbf{(hfg*)}}, we know  $E_2^{*,*}(A)=\coh^*(H,\coh^*(R,k))$ is a finitely generated and noetherian algebra over $E_2^{0,0}(A)=k$ and it is a finite module over its graded center. Moreover since $\coh^*(R,M)$ is finite over $\coh^*(R, k)$, Lemma~\ref{lem:hfg}(iv) implies that $E_2^{*,*}(M)=\coh^*(H,\coh^*(R,M))$ is a finite module over $E_2^{*,*}(A)=\coh^*(H,\coh^*(R,k))$. Thus the result holds in characteristic zero and for positive characteristic we can apply Lemma \ref{lem:key} since all the assumptions there are satisfied when $r_0=2$. 
\end{proof}


\noindent
{\bf Hochschild cohomology of smash products:} Now let $A=R\#H$ be the smash product of a finite dimensional Hopf algebra $H$ and a finite dimensional $H$-module algebra $R$. There are Lyndon-Hochschild-Serre spectral sequences associated to the smash product $A=R \# H$ and any finite bimodule $M$ over $A$: 
\begin{align*}
E_2^{p,q}(A)=\coh^p(H,\HH^q(R,A)) &~\Longrightarrow~ \HH^{p+q}(A)=E_\infty^{p,q}(A),\\
E_2^{p,q}(M)=\coh^p(H,\HH^q(R,M)) &~\Longrightarrow~ \HH^{p+q}(A,M)=E_\infty^{p,q}(M).
\end{align*}
These spectral sequences naturally inherit structures that are compatible with the multiplicative structure of $\HH^*(A)$ and its module structure on $\HH^*(A,M)$ (see Appendix~\ref{subsec:hochschild}). 

\begin{theorem}\label{thm:FiniteTypeHochschild}
Retain the notations above. Further assume that $R$ satisfies {\rm \textbf{(fg)}} and $H$ is cocommutative. Then $A=R \# H$ satisfies {\rm \textbf{(fg)}} if the corresponding LHS spectral sequence $E^{*,*}_r(A)$ collapses for
\begin{enumerate}
\item $r=2$ when $\chara(k)=0$, or 
\item some $r\ge 2$ when $\chara(k)>0$.
\end{enumerate}
\end{theorem}
\begin{proof} 
Since $R$ satisfies {\rm \textbf{(fg)}}, $\HH^*(R,A)$ is a finite module over $\HH^*(R)$. This implies that $\HH^*(R,A)$ is finitely generated and noetherian since $\HH^*(R)$ is. Furthermore, $\HH^*(R)$ is mapped to the graded center of $\HH^*(R,A)$ via the embedding $R\hookrightarrow A$ by Proposition~\ref{prop:fg}(ii). Next we show that the image of $\HH^*(R)$ is an $H$-module subalgebra of $\HH^*(R,A)$. 

Using the notation in Appendix~\ref{subsec:hochschild}, let $K_{\bu}$ be an $H$-equivariant $R$-bimodule resolution of $R$. Then $\HH^*(R,A) \cong \coh^*(\Hom_{R^e}(K_{\bu},A))$ and $\HH^*(R) \cong \coh^*(\Hom_{R^e}(K_{\bu},R))$. Choose any $f \in \Hom_{R^e}(K_{\bu},R) \subset \Hom_{R^e}(K_{\bu},A)$, $h \in H$, and $u \in K_{\bu}$. Applying the $H$-action given in Appendix~\ref{subsec:hochschild}
and the assumption that $H$ is cocommutative, we find that
\begin{align*}
(f \cdot h)(u) &~=~\sum S(h_1) f(h_2 \cdot u) h_3 ~=~\sum [S(h_2) \cdot f(h_3 \cdot u)] S(h_1)h_4 \\
&~=~ \sum [S(h_3) \cdot f(h_4 \cdot u)] S(h_1)h_2 ~=~ \sum S(h_1) \cdot f(h_2 \cdot u) \in B. 
\end{align*}
Hence, $\Hom_{R^e}(K_{\bu},R)$ is an $H$-invariant subcomplex of $\Hom_{R^e}(K_{\bu},A)$. Passing to homology, the image of $\HH^*(R)$ is an $H$-module subalgebra of $\HH^*(R,A)$. 

As a result of Friedlander and Suslin~\cite{FS}, $H$ satisfies {\rm \textbf{(hfg)}} and hence it satisfies {\rm \textbf{(hfg*)}} by Lemma~\ref{lem:hfg} since $H$ is cocommutative (see~Proposition \ref{prop:integral}). By Proposition~\ref{prop:hfg}(ii), we know $E_2^{*,*}(A)=\coh^*(H,\HH^*(R,A))$ is a finitely generated noetherian algebra over the center $E_2^{0,0}(A)=Z(A)$ of $A$ and it is a finite module over its graded center. Moreover since $\HH^*(R,M)$ is a finite module over $\HH^*(R,A)$, Lemma~\ref{lem:hfg}(iv) implies that $E_2^{*,*}(M)=\coh^*(H,\HH^*(R,M))$ is finitely generated over $E_2^{*,*}(A)=\coh^*(H,\HH^*(R,A))$. Thus, the result holds in characteristic zero. For positive characteristic, we can apply Lemma \ref{lem:key} since all the assumptions there are satisfied when $r_0=2$. 
\end{proof}


\section{A lifting method for the finite generation conditions via reduction modulo $p$}
\label{sec:mod p}

\numberwithin{equation}{section}

In this section, we use the reduction modulo $p$ method in number theory to deal with the finite generation conditions for the (Hochschild) cohomology ring of (Hopf) algebras over a field of characteristic zero, or mainly over the field of complex numbers $\mathbb C$.

Our first result suggests that regarding the finite generation conditions over a field of characteristic zero, it suffices to work over $\mathbb C$.

\begin{lemma}\label{lem:CN}
Let $k$ be an arbitrary field of characteristic zero, and $A$ be a finite dimensional algebra (resp. Hopf algebra) over $k$. Then there is some finite dimensional algebra $A'$ (resp. Hopf algebra) over $\mathbb C$, which is obtained from a field extension of some subring of $A$, such that $A$ satisfies {\rm (\textbf{fg})} (resp. {\rm (\textbf{hfg})}) if and only if $A'$ satisfies {\rm (\textbf{fg})} (resp. {\rm (\textbf{hfg})}). 
\end{lemma}
\begin{proof}
Here we only treat the case when $A$ is a finite dimensional algebra over $k$ with condition {\rm (\textbf{fg})}. The argument for $A$ being a Hopf algebra with condition {\rm (\textbf{hfg})} is similar. Since $k$ is of characteristic zero, the prime field of $k$ is $\mathbb Q$. Fix a finite basis $x_1,\ldots,x_n$ of $A$. We can write the multiplication in $A$ as $x_ix_j=\sum_\ell \alpha_{ij}^\ell x_\ell$ for some $\alpha_{ij}^\ell\in k$. Take $K=\mathbb Q(\alpha_{ij}^\ell)$ to be the subfield of $k$ by joining all the coefficients $\alpha_{ij}^\ell$ to its prime field $\mathbb Q$. Then one can define another algebra $B$ over $K$ with the bases $x_i$'s and the same multiplication rule $x_ix_j=\sum_\ell \alpha_{ij}^\ell x_\ell$. It is clear that $B\otimes_Kk\cong A$. Now since $\mathbb C$ is algebraically closed over $\mathbb Q$ and has infinitely many transcendental numbers, one can embed $K$ into $\mathbb C$. Let $A':=B\otimes_K\mathbb C$ which is a finite dimensional complex algebra. By Lemma \ref{Rem:ext} and Remark \ref{rem:FE}, $A$ satisfies {\rm (\textbf{fg})} $\Leftrightarrow B$ satisfies {\rm (\textbf{fg})} $\Leftrightarrow A'$ satisfies {\rm (\textbf{fg})}.
\end{proof}

\begin{lemma}\label{lem:field}
Let $R=\bigoplus_{i\ge 0}R_i$ be a connected graded commutative algebra over a base field $K\subset \mathbb C$. If $R\otimes_Kk$ is noetherian for some field extension $k/K$, then $R\otimes_K\mathbb C$ is noetherian. 

Moreover, if $M=\bigoplus_{i\ge 0}M_i$ is a graded module over $R$ such that $M\otimes_Kk$ is finite over $R\otimes_Kk$, then $M\otimes_K\mathbb C$ is finite over $R\otimes_K\mathbb C$.  
\end{lemma}

\begin{proof}
By Proposition \ref{prop:1}(ii), we know $R\otimes_Kk$ is finitely generated over $k$, say by homogenous elements $f_1,\dots,f_r$. Choose a $K$-basis $\{x_i\}$ for $R$. Denote by $F$ the subfield of $k$ by joining all the coefficients appearing in $f_1,\dots,f_r$ as $k$-linear combinations of the basis $\{x_i\}$ to $K$. Then $f_1,\dots,f_r$ belong to $R\otimes_K F$ and they generate a $F$-subalgebra in $R\otimes_K F$, which is denoted by $T$. Note that $R\otimes_Kk$ is locally finite and has a Hilbert series. It is clear that $T\otimes_Fk\cong R\otimes_Kk$. Since base field extension does not change the Hilbert series, $T\subseteq R\otimes_KF\subseteq R\otimes_Kk$ share the same Hilbert series. This implies that $T=R\otimes_K F$ and $R\otimes_K F$ is finitely generated. Now by the fact that $\mathbb C$ is algebraically closed over $K$ and has uncountably many transcendental numbers, we can embed $F$ into $\mathbb C$. Hence $(R\otimes_K F)\otimes_{F}\mathbb C\cong R\otimes_K\mathbb C$. So $R\otimes_K\mathbb C$ is finitely generated and hence is noetherian by Proposition \ref{prop:1}(ii).

Finally for any graded module $M$ over $R$, take finitely many homogenous generators $f_1,\cdots,f_r$ of $M\otimes_Kk$ over $R\otimes_Kk$. Then there is a middle field $K\subseteq F\subseteq k$ such that $f_1,\dots,f_r$ belong to $M\otimes_KF$. By a similar argument to that above, we can consider $T$ as the submodule generated by $f_1,\dots,f_r$ in $M\otimes_KF$ and conclude that $T\otimes_F \mathbb C\cong M\otimes_K \mathbb C $ is finitely generated over $R\otimes_K\mathbb C$ via some embedding of $F$ into $\mathbb C$.  
\end{proof}

\begin{lemma}\label{lem:completion}
Let $R=\bigoplus_{i\ge 0}R_i$ be a graded commutative algebra over a noetherian commutative base ring $R_0$ that is locally finite over $R_0$, and $I$ be any ideal of $R_0$ . If $R\otimes_{R_0}(R_0/I)$ is noetherian, then $R\otimes_{R_0}(\varprojlim\limits_{i}R_0/I^i)$ is noetherian. 

Moreover, let $M=\bigoplus_{i\ge 0}M_i$ be a graded module over $R$ that is locally finite over $R_0$. If $M\otimes_{R_0}(R_0/I)$ is finite over $R\otimes_{R_0}(R_0/I)$, then $M\otimes_{R_0}(\varprojlim\limits_{i}R_0/I^i)$ is finite over $R\otimes_{R_0}(\varprojlim\limits_{i}R_0/I^i)$.
\end{lemma}

\begin{proof}
We write $\widehat{R_0}=\varprojlim\limits_{i}R_0/I^i$ and $\widehat{R}=R\otimes_{R_0}\widehat{R_0}$. For any $R_0$-module $M$, we denote the natural map
\[
\xymatrix{
\varphi_N: M\otimes_{R_0} \widehat{R_0}\ar[r]&\varprojlim\limits_{i} M/I^iM.
}
\]
Suppose $M$ is finite over $R_0$. By taking a finite presentation $R_0^m\to R_0^n\to M\to 0$ of $M$, we get the following commutative diagram:
\[
\xymatrix{
 \widehat{R_0}^m\ar[r]\ar[d]^-{\varphi_{R_0^m}} & \widehat{R_0}^n\ar[r]\ar[d]^-{\varphi_{R_0^n}}  & M\otimes_{R_0}  \widehat{R_0}\ar[d]^-{\varphi_M}\ar[r]& 0 \\
 \varprojlim\limits_{i} (R_0/I^i)^m\ar[r] &\varprojlim\limits_{i} (R_0/I^i)^n\ar[r] & \varprojlim\limits_{i} M/I^iM\ar[r]& 0
}  
\]
The first row above is exact since it is obtained by applying $-\otimes_{R_0} \widehat{R_0}$ to the finite presentation of $M$. Moreover, we know ${\varprojlim\limits_{i}}^1 R_0/I^i=0$ for the maps in the inverse sequence 
\[
\xymatrix{
R_0/I & R_0/I^2\ar@{->>}[l] &  R_0/I^3\ar@{->>}[l]  & \ar@{->>}[l] \cdots
}
\]
satisfy the Mittag-Leffler condition. Then by taking the inverse limit of the exact sequence $(R_0/I^i)^r\to (R_0/I^i)^s\to  M/I^iM\to 0$ over $i\ge 1$, we get the second row above is exact. Now since $\varphi_{R_0^r}$ and $\varphi_{R_0^r}$ are isomorphisms, we know $\varphi_M$ is an isomorphism whenever $M$ is finite over $R_0$. As a consequence since $R$ is locally finite over $R_0$, we have 
$$
\widehat{R}~=~\left(\bigoplus_{j\ge 0} R_j\right)\otimes_{R_0}\widehat{R_0}~\cong~\bigoplus_{j\ge 0}\left( R_j\otimes_{R_0}\widehat{R_0}\right)~\cong~\bigoplus_{j\ge 0} \left(\varprojlim\limits_{i} R_j/I^iR_j\right).
$$ 
So $\widehat{R}$ has an $I$-adic filtration such that 
\begin{align*}
\gr_I \widehat{R}~=~\bigoplus_{i\ge 0} \left(I^i \widehat{R}\big/ I^{i+1} \widehat{R}\right)&~\cong~\bigoplus_{j\ge 0}\left(\bigoplus_{i\ge 0} I^iR_j/I^{i+1}R_j\right)~\cong~\bigoplus_{j\ge 0}\left(\bigoplus_{i\ge 0} R_j\otimes_{R_0} I^i/I^{i+1}\right)\\&~\cong~ R\otimes_{R_0}\left(\bigoplus_{i\ge 0} I^i\big/I^{i+1}\right)~\cong~ R\otimes_{R_0} \gr_I \widehat{R_0}.
\end{align*}
Since $R_0$ is noetherian, $I$ is finitely generated, say $I=(a_1,\dots,a_n)$. Thus there is a surjection $(R_0/I)[t_1,\dots,t_n]\twoheadrightarrow \gr_I \widehat{R_0}$ given by $t_i\mapsto a_i$ in $I/I^2$. Then we get a  surjection $(R\otimes_{R_0} R_0/I)[t_1,\dots,t_n]\twoheadrightarrow \gr_I \widehat{R}$. By assumption we know $R\otimes_{R_0}(R_0/I)$ is noetherian. Hence $\gr \widehat{R}$ is noetherian and so is $\widehat{R}$. The statement for finite generation of modules can be proved similarly. 
\end{proof}

\begin{definition}\label{D:Roder}
Let $A$ be a finite dimensional Hopf algebra over $\mathbb C$. We say $A$ can be defined over some algebraic number field $K/\mathbb Q$ if there is some Hopf $K$-subalgebra $B$ of $A$ such that $A\cong B\otimes_K\mathbb C$.  
\end{definition}

\medskip
\noindent
{\bf Note added for arXiv version 3.} Theorem 5.5 in the published version as originally stated is incorrect. In this version, we provide a corrected statement for Theorem~\ref{thm:m1} below by adding the Tor vanishing hypothesis \eqref{vanishTor}, following our corrigendum to the published paper. Theorem~\ref{thm:m2}, Proposition~\ref{prop:m1} and the example following it are adjusted accordingly.
\medskip

\begin{theorem}
\label{thm:m1}
Let $A$ be a finite dimensional complex Hopf algebra that can be defined over some algebraic number field $K$. Assume that for every $p \geq 0$,
\begin{equation} 
\label{vanishTor}
\Tor_{1}^{\mathcal R} (\coh^{p} (B, \mathcal R), F) = 0 \quad \text{ and } \quad \Tor_{1}^{\mathcal R} (\coh^{p} (B, W), F) = 0,
\end{equation} 
where $\mathcal R$, $B$, $W$ and $F$ are as constructed in the proof below.
Then there exists a finite dimensional Hopf algebra $A'$ over some finite field $F$ such that if $A'$ satisfies {\rm (\textbf{hfg})} then so does $A$. 

Moreover, $A$ can be viewed as a deformation of $A'$ in the following way: there exists a free Hopf $\mathcal R$-subalgebra $B$ of $A$ such that $A\cong B\otimes_{\mathcal R}\mathbb C$, where we can let $A'= B\otimes_{\mathcal R}F$ with $F = \mathcal R/{\mathfrak m}$ a finite field, for a maximal ideal ${\mathfrak m}$ of $\mathcal R$. 
\end{theorem}

\begin{proof}
Since $A$ can be defined over $K$, there exists some Hopf $K$-subalgebra, denoted by $L$, such that $L\otimes_K\mathbb C\cong A$. In view of Lemma \ref{Rem:ext}, by replacing $K$ with a possible finite field extension, we may assume the Jacobson radical $J_L$ of $L$ splits in $L$. Thus we can choose a finite basis $x_1,\ldots,x_n$ for $L$ such that $J_L={\rm span}_K(x_1,\ldots,x_m)$ and $J_L\otimes_K\mathbb C\cong J_A$, the Jacobson radical of $A$. 

Let $\mathcal S$ be the subset of $K$ which consists of all the coefficients when we apply (co)multiplication, (co)unit, and antipode of $L$ to the above fixed basis. Let $\mathcal R$ be the localization of $\mathcal O_K$ at the multiplicative set generated by all the denominators appearing in $\mathcal S$. Specifically, one can choose a nonzero $d \in \mathcal O_K$ clearing the finitely many denominators in the structure constants and set 
$$ \mathcal R = \mathcal O_K [1/d], \qquad B = \bigoplus_{i=1}^n \mathcal R x_i, \qquad W = \bigoplus_{j=1}^m \mathcal R x_j. $$
Since $\dim_K L<\infty$, $\mathcal S$ is finite and $\mathcal R$ is a finitely generated $\mathbb Z$-algebra. One can check that $B$ is a free Hopf $\mathcal R$-subalgebra of $A$ satisfying 
$$B\otimes_\mathcal R\mathbb C~\cong~(B\otimes_\mathcal RK)\otimes_K\mathbb C~\cong~L\otimes_K\mathbb C~\cong~A.$$
Moreover, $W$ is a finite $B$-module satisfying 
$$W\otimes_\mathcal R\mathbb C~\cong~(W\otimes_\mathcal RK)\otimes_K\mathbb C~\cong ~J_L\otimes_K\mathbb C~\cong~J_A.$$

We use the language of differential graded algebras and modules to describe the cohomology ring $\coh^*(B, \mathcal R)$ of $B$ and the cohomology $\coh^*(B,W)$ of $B$ with coefficients in $W$. We apply $-\otimes_B\mathcal R$ to the bar resolution of $B$ over $\mathcal R$,
\[
\xymatrix{
\cdots\ar[r]^-{\partial_3}&B\otimes_{\mathcal R} B\otimes_{\mathcal R} B\otimes_{\mathcal R} B\ar[r]^-{\partial_2}& B\otimes_{\mathcal R} B \otimes_{\mathcal R} B \ar[r]^-{\partial_1} & B\otimes_{\mathcal R} B \ar[r]^-{\mu}& B,
}
\]
where $\partial_i(b_0\otimes \cdots \otimes b_{i+1})=\sum_{j=0}^i(-1)^jb_0\otimes \cdots \otimes b_jb_{j+1}\otimes \cdots \otimes b_{i+1}$. Since $B$ is free over $\mathcal R$, we get a projective resolution of $\mathcal R$ in the category of left $B$-modules. Therefore, we obtain two complexes by applying $\Hom_B(-,\mathcal R)$ and $\Hom_B(-,W)$ to this resolution, namely   
\begin{align}\label{E:complexB}
C^n(B,\mathcal R)~:=~\Hom_{\mathcal R}(B^{\otimes n},\mathcal R) \quad \text{and} \quad C^n(B,W)~:=~\Hom_{\mathcal R}(B^{\otimes n}, W),
\end{align}
where we omit the formulas of the corresponding differentials. Note that $C^\bullet(B,\mathcal R)$ is a differential graded algebra with the cup product given by, for any $f\in C^m(B,\mathcal R)$ and $g\in C^n(B,\mathcal R)$,
\begin{align}\label{E:dgm}
f\smile g(b_1\otimes \cdots b_{m+n})~=~f(b_1\otimes \cdots b_m)g(b_{m+1}\otimes \cdots b_{m+n}).
\end{align}
Moreover, $C^\bullet(B,W)$ is differential graded module over $C^\bullet(B,\mathcal R)$ with the module action similar to \eqref{E:dgm} by considering $g\in C^n(B, W)$. After taking the cohomology of the two complexes in \eqref{E:complexB}, we obtain $\coh^*(B,\mathcal R)$ and $\coh^*(B,W)$, where the Yoneda product is compatible with the cup product. 

Fix a nonzero maximal ideal $\mathfrak{m}$ of $\mathcal R$ and set 
$$F=\mathcal R / \mathfrak{m}, \qquad A' = B \otimes_{\mathcal R} F, \qquad W' = W \otimes_{\mathcal R} F.$$
It is clear from the construction that $F$ is a finite field and $W'$ is a finite module over $A'$. 

Since $B$ is a free module over $\mathcal R$ of finite rank, it is projective and so is $B^{\otimes n}$. Hence, we have
\begin{align*}
C^n(B,\mathcal R)\otimes_{\mathcal R}F&\, =\Hom_{\mathcal R}(B^{\otimes n},\mathcal R)\otimes_{\mathcal R} F\cong \Hom_F((B^{\otimes n})\otimes_{\mathcal R}F,F)\\
&\,\cong \Hom_F((B\otimes_{\mathcal R}F)^{\otimes n},F)=\Hom_F(A'^{\otimes n},F)=:C^n(A',F),
\end{align*}
which calculates the cohomology ring $\coh^*(A',F)$. Likewise, $C^n(B,W)\otimes_{\mathcal R}F=\Hom_F(A'^{\otimes n},W')=:C^n(A',W')$. By the hypothesis that $A'$ satisfies {\rm (\textbf{hfg})}, we know $\coh^*(C^\bullet(A',W'))$ is noetherian over $\coh^*(C^\bullet(A',F))$.

Now as a localization of $\mathcal O_K$, $\mathcal R$ is a Dedekind domain and hence is hereditary. So higher $\Tor$ groups over $\mathcal R$ also vanish; in particular, the assumption \eqref{vanishTor} gives 
$$\Tor^{\mathcal R}_{\geq 1}(\coh^p(C^{\bu}(B,\mathcal R)), F) = 0 \quad \text{ and } \quad \Tor^{\mathcal R}_{\geq 1}(\coh^p(C^{\bu}(B,W)), F) = 0,$$ 
for all $p \geq 0$. Moreover, we have a projective resolution $0 \to \mathfrak{m} \to \mathcal R\to F\to 0$ in the category of left $\mathcal R$-modules. 
After tensoring the resolution with the complex $C^{\bu}(B,\mathcal R)$, we get a double complex which yields a spectral sequence 
$$
E_2^{pq}:=\Tor^{\mathcal R}_q\left(\coh^p(C^{\bu}(B,\mathcal R),F)\right)~\Rightarrow~E_\infty^{pq}:=\coh^{p+q}\left(C^{\bu}(B,\mathcal R)\otimes_{\mathcal R}F\right)\cong \coh^{p+q}(A',F),
$$
which collapses at the $E_2$ page. Under the assumption \eqref{vanishTor}, the $E_2$ page of this spectral sequence is simply $\coh^p(C^{\bu}(B,\mathcal R))\otimes_{\mathcal R} F$ concentrated on the $p$-axis (that is, $q=0$). In particular, we obtain the natural isomorphisms
\[ \coh^*(B, \mathcal R) \otimes_{\mathcal R} F \cong \coh^*(A',F), \qquad \coh^*(B, W) \otimes_{\mathcal R} F \cong \coh^*(A',W') \]
compatible with the cup product and module action. Since $A'$ satisfies {\rm (\textbf{hfg})} by hypothesis, it follows that $\coh^{*}(B,\mathcal R)\otimes_{\mathcal R} F$ is noetherian and $\coh^{*}(B,W)\otimes_{\mathcal R} F$ is a finite module over it. 

Note that ${\rm Frac}\,\mathcal R={\rm Frac}\,\mathcal O_K=K$. Let $T:=\coh^{*}(B,\mathcal R)\otimes_{\mathcal R}K $ and $M:=\coh^{*}(B,W)\otimes_{\mathcal R}K$. By previous discussion, we have $T\otimes_Kk$ is noetherian and $M\otimes_Kk$ is finite over $T\otimes_Kk$. Note that $K\subset \mathbb C$ is a subfield. By Lemma \ref{lem:field}, one sees that $T\otimes_K\mathbb C$ is noetherian and $M\otimes_K\mathbb C$ is finite over $T\otimes_K \mathbb C$. Finally, since $\mathbb C$ is flat over $K={\rm Frac}\, \mathcal R$ which is also flat over $\mathcal R$, we have 
$$T\otimes_{\mathcal R} \mathbb C~=~\coh^*\left(B,\mathcal R\right)\otimes_{\mathcal R}\mathbb C~\cong~\coh^*\left(B,\mathcal R\otimes_{\mathcal R}\mathbb C\right)~\cong~\coh^*\left(C^{\bu}(A,\mathbb C)\right)~\cong~\coh^*(A,\mathbb C)$$
and
$$M\otimes_{\mathcal R} \mathbb C~=~\coh^*\left(B,W \right)\otimes_{\mathcal R}\mathbb C~\cong~\coh^*\left(B,W \otimes_{\mathcal R}\mathbb C\right)~\cong~\coh^*\left(C^{\bu}(A,J_A)\right)~\cong ~\coh^*(A,J_A).$$
Then we conclude that $A$ satisfies {\rm \textbf{(hfg)}} by Proposition \ref{equivfg*}. 
\end{proof}

The following result can be proved analogously. 

\begin{theorem}\label{thm:m2}
Let $A$ be a finite dimensional complex algebra that can be defined over some algebraic number field. Assume that an analogue of the Tor vanishing hypothesis \eqref{vanishTor} in Theorem~\ref{thm:m1} holds for the Hochschild cohomology setting. Then there exists a finite dimensional algebra $A'$ over some finite field, where $A$ can be viewed as some deformation of $A'$, such that if $A'$ satisfies {\rm (\textbf{fg})}, then so does $A$.
\end{theorem}

\begin{remark}
A similar proof works for finite dimensional augmented algebras that can be defined over some algebraic number field. Moreover, we can always assume the resulting (Hopf or augmented) algebra after the reduction modulo $p$ is over an algebraically closed field of characteristic $p$ by a field extension in view of Lemma \ref{Rem:ext} and Remark \ref{rem:FE}. The analogous additional Tor vanishing hypothesis, as in Theorem~\ref{thm:m1}, is needed in these cases as well. 
\end{remark}

\noindent
\underline{{\bf Applications.}} We provide here some applications for the finite generation conditions when the resulting Hopf algebra via reduction modulo $p$ is the smash product of a quantum complete intersection with a semisimple Hopf algebra. In particular, our result is applicable for finite dimensional pointed Hopf algebras of diagonal type over the complex numbers.

\begin{lemma}\label{lem:SP}
Let $R$ be a finite dimensional (resp.\ augmented) $k$-algebra satisfying {\rm \textbf{(fg)}} (resp.\ {\rm \textbf{(hfg)}}), and $H$ a semisimple Hopf algebra over $k$. Suppose there is an $H$-action on $R$ (resp.\ preserving the augmentation of $R$).  Then the smash product $R \# H$ satisfies {\rm \textbf{(fg)}} (resp.\ {\rm \textbf{(hfg)}}).
\end{lemma}
\begin{proof}
Here we prove the augmented case, and the argument for the algebra case is similar. Since $H$ is semisimple, we know $\coh^*(R\#H,k)\cong \coh^*(R,k)^H$. By hypothesis, $\coh^*(R,k)$ is finitely generated and noetherian, and so is its invariant ring $\coh^*(R,k)^H$ by \cite[Corollary 4.3.5]{MO93}. Let $M$ be a finite $R\# H$-module. Since $M$ is also finite over $R$ by the restriction, we have $\coh^*(R,M)$ is noetherian over $\coh^*(R,k)$ by hypothesis. Thus the submodule $\coh^*(R\#H,M)\cong \coh^*(R,M)^H$ of $\coh^*(R,M)$ is finite over $\coh^*(R,k)$. Then we can conclude that $\coh^*(R\#H,M)$ is finite over $\coh^*(R\#H,k)$ as $\coh^*(R,k)$ is finite over $\coh^*(R,k)^H$ by \cite[Theorem 4.4.2]{MO93}.
\end{proof}

We will work with some {\bf quantum complete intersections} $R$, as recalled here. 

Let $t$ be a positive integer and for each $1\leq i\leq t$, 
let $N_i\ge 2$ be an integer. Let $q_{ij}\in k^{\times}$ for $1\leq i<j\leq t$.
Let $R$ be the $k$-algebra generated by $x_1,\ldots, x_t$,
subject to relations
\[
   x_i x_j = q_{ij} x_j x_i \quad \mbox{ and } \quad 
   x_i^{N_i} = 0,
\] 
for all $i<j$ and all $i$. It is clear that $R$ is augmented since $R$ is local with the unique maximal ideal $(x_1,\ldots,x_t)$. 

The cohomology ring $\coh^*(R,k)$ is finitely generated and noetherian. See~\cite[Theorem~5.3]{BO08} and~\cite[Theorem~4.1]{MPSW}.
The latter reference corrects some small errors in the relations of
the former, but omits the necessary distinction of the cases
where $N_i =2$ for some $i$.
(The proof of exactness of the resolution in~\cite{MPSW}
requires characteristic~0, however it is essentially the same
resolution as that given in~\cite{BO08}, which is proven to be
exact in any characteristic.) So in this case, we know $R$ satisfies {\rm \textbf{(hfg)}} by Proposition \ref{equivfg*} since $R$ is local. 

Hochschild cohomology behaves somewhat differently:
$\HH^*(R)$ itself is in fact finite dimensional for some
choices of values of $q_{ij}$.
By~\cite[Theorem~5.5]{BO08} and Proposition~\ref{equivfg}, 
$\HH^*(R)$ satisfies the finite generation condition {\rm \textbf{(fg)}}
if and only if all $q_{ij}$ are roots of unity, see~\cite{BGMS}.

\begin{prop}\label{prop:m1}
Let $A$ be a finite dimensional complex (resp.\ Hopf) algebra that can be defined over some algebraic number field. If the constructed finite dimensional (resp.\ Hopf) algebra via reduction modulo $p$ over some finite field is the smash product of a quantum complete intersection with a semisimple Hopf algebra, and the analogue of Tor vanishing condition in Theorem~\ref{thm:m2} (resp.\ Tor vanishing condition \eqref{vanishTor} in Theorem~\ref{thm:m1}) holds, then $A$ satisfies {\rm \textbf{(fg)}} (resp.\ {\rm \textbf{(hfg)}}). 
\end{prop}
\begin{proof}
It follows from Lemma \ref{lem:SP} and the discussion above, where we notice that any nonzero number in a finite field is a root of unity. 
\end{proof}

Our lifting method of reduction modulo $p$ can be applied to finite dimensional pointed Hopf algebras of diagonal type over the complex numbers. 

\begin{example}
Let $A$ be a finite dimensional pointed Hopf algebra of diagonal type over the field of complex numbers. We follow the strategy of Andruskiewitsch-Schneider in \cite{AS, AS2}. Denote the coradical of $A$ by $\mathbb C[G]$ for some finite group $G$. Then the associated graded algebra $\gr A$ is isomorphic to a smash product algebra $\mathcal B(V)\#\mathbb C[G]$ with respect to its coradical filtration, where $\mathcal B(V)$ is the Nichols algebra of some Yetter-Drinfeld module $V$ in ${}^G_G{\mathcal YD}$. Since the braided space $(V,c)$ is of diagonal type, there is a basis $\{x_1,\dots,x_\theta\}$ of $V$ and a collection of scalars $(q_{ij})_{1\le i,j\le \theta}$ such that $c(x_i\otimes x_j)=q_{ij}x_j\otimes x_i$ for all $1\le i,j\le \theta$. Notice that the coefficients in the braiding matrix $\mathbf{q}=(q_{ij})_{1\le i,j\le \theta}$ are all algebraic numbers since they all come from certain characters of the finite group $G$. Then one can show that $\gr A$ can be defined over the algebraic number field $\mathbb Q(q_{ij})$. Moreover, the Drinfeld double of $\gr A$, denoted by $\mathcal D(\gr A)$, also can be defined over the algebraic number field $\mathbb Q(q_{ij})$. An explicit presentation of $\mathcal D(\gr A)$ can be found in \cite[Lemma 7]{PV16}. 

Now suppose the constructed finite dimensional Hopf algebra from $\mathcal D(\gr A)$ via reduction modulo $p$ is the smash product of a quantum complete intersection with a semisimple Hopf algebra (e.g., $\mathcal D(G)$ by a careful choice of $p$ such that $p\nmid |G|$). If the hypotheses of Proposition \ref{prop:m1} hold, then we know  $\mathcal D(\gr A)$ satisfies {\rm \textbf{(hfg)}}. Moreover by Masuoka's result \cite{Ma08}, $A$ is always some 2-cocycle twist of the associated graded algebra $\gr A$. (Also see Angiono and Garcia-Iglesias's survey paper \cite[\S 1.2.1]{AG}.) As a consequence, $A$ can be embedded into the Drinfeld double $\mathcal D(\gr A)$. So we can conclude that $A$ satisfies {\rm \textbf{(hfg)}} by Proposition \ref{prop:ext} in this case.

However, note that in practice in general, we do not know whether the constructed finite dimensional Hopf algebra $A'$ from $\mathcal{D}({\rm gr}\, A)$ via reduction modulo $p$ satisfies the Tor vanishing assumption \eqref{vanishTor}. Therefore, we cannot apply Proposition~\ref{prop:m1} to conclude that $\mathcal{D}({\rm gr}\, A)$ satisfies {\rm \textbf{(hfg)}}, and thus cannot conclude from this proposition that the finite dimensional pointed Hopf algebra $A$ satisfies {\rm \textbf{(hfg)}}. Different techniques are used in~\cite{AAPW}, for example, to prove cohomological finite generation for these Hopf algebras.
\end{example}


\appendix
\numberwithin{equation}{subsection}
\section{Some relevant spectral sequences related to cohomology} 
\label{appendix}

For completeness, we provide here explicit descriptions of the spectral sequences related to the (Hochschild) cohomology of filtered algebras, smash products and crossed products. We follow standard arguments to show that these spectral sequences are multiplicative, that is, they preserve the ring structure of the (Hochschild) cohomology.


\subsection{Cohomology of a filtered augmented algebra}
\label{subsec:May}
We recall May's spectral sequence~\cite{M} for the cohomology of
an augmented filtered algebra.
Let $A$ be an augmented algebra with a finite algebra filtration,
either increasing or decreasing:
\[
  0  =  A_{-1} \subset A_0\subset A_1\subset \cdots \subset A_m = A
      \quad \quad
 \mbox{ or } \quad \quad   A  =  A_0 \supset A_1 \supset A_2\supset \cdots \supset A_m = 0 
  , 
\]
so that $A_i A_j \subset A_{i+j}$ for all $i,j$. 
Assume the augmentation $\varepsilon$ on $A$ satisfies $\varepsilon(A_n)=0$ for all 
$n>0$ in the increasing filtration case 
(in the decreasing filtration case, assume $\varepsilon(A_1)=0$), so that it 
induces an augmentation on the associated graded algebra
$\gr A = \bigoplus_{n \ge 0} (\gr A)_n$.
For example, if 
$A$ is a Hopf algebra,
we will be interested in the cases where either the coradical filtration
(increasing) or the Jacobson radical filtration (decreasing)
is in fact a Hopf algebra filtration.

The filtration on $A$ induces a filtration on the reduced bar resolution
of $k$ as an $A$-module: 
\begin{equation}
\label{pdot}
  P_{\bu}: \quad \quad \cdots \stackrel{d_3}{\longrightarrow}
   A\ot \bar{A}^{\ot 2}\stackrel{d_2}{\longrightarrow}
   A\ot \bar{A}\stackrel{d_1}{\longrightarrow}
   A\stackrel{\varepsilon}{\longrightarrow}k \rightarrow 0,
\end{equation}
where $\bar{A} = A/ k\cdot 1_A$ (a vector space quotient). We will give details here for an increasing filtration.
A decreasing filtration leads similarly to a spectral sequence.
The increasing filtration on $P_{\bu}$ induced
by that on $A$ is given in each degree $n$ by
$$
  F_i (A\ot\bar{A}^{\ot n})~=~\sum_{i_0 + \cdots + i_n = i} F_{i_0}A
     \ot F_{i_1}\bar{A}\ot \cdots\ot F_{i_n}\bar{A}.
$$
The reduced bar resolution of $k$ as a module over 
$\gr A$ is precisely $\gr P_{\bu}$, where
 $ \gr_i P_n ~:=~ F_iP_n/F_{i-1}P_n $.
Now let $C^{\bu} = C^{\bu}(A):= \Hom_A(P_{\bu},k)$.
Then $C^n(A)$
is a filtered vector space where
$$
  F^iC^n(A)~ =~ \{ f\in\Hom_A(P_n,k) :  f |_{F_{i-1}P_n} = 0\} .
$$
(In the decreasing filtration case, 
$F^i C^n (A) = \{ f\in\Hom_A(P_n,k):  f |_{F_{i+1}P_n } = 0\}$.)
This filtration is compatible with the differentials on $C^{\bu}(A)$. 
Hence, $C^{\bu}(A)$ is a filtered cochain complex. (A decreasing filtration on $A$ similarly leads to an increasing filtration on $C^{\bu}$).
By construction, if $A$ is finite dimensional, 
the filtration is  both bounded above and below. There is a cohomology spectral sequence associated to the filtration on $C^{\bu}$ (see~\cite[5.4.1 and 5.5.1]{W}), 
$$E_0^{i,j} = \frac{F^iC^{i+j}}{F^{i+1}C^{i+j}}, \qquad E_1^{i,j} = \coh^{i+j}(E_0^{i,*}) \cong \coh^{i+j}( \gr _i A,k),$$ 
converging to the cohomology of $C^{\bu}$:
$   \coh^*(\gr_* A,k)~ \Longrightarrow~ \coh^*(A,k).
$
The spectral sequence is multiplicative by its definition.
Thus $E_{\infty}$ is the associated graded
algebra of $\coh^*(A,k)$. 
(See \cite[4.5.2, 5.2.13, and 5.4.8]{W}.
See also~\cite[Theorem 12.5]{McC}.)
If $M$ is any $A$-module, there is an induced filtration with
components $A_iM$ in either case (increasing or decreasing filtration), and 
there is similarly a spectral sequence
$\widetilde{E} := \widetilde{E}(M)$ associated to the filtration on 
$C^{\bu}(A,M) = \Hom_A(P_{\bu},M)$ given by (in the increasing 
filtration case): 
\[
    F^i C^n (A,M)~ =~ \{ f\in\Hom_A(P_n,M):
     f(F_\ell P_n)\subset F_{\ell-i} M \text{ for all } \ell \} .
\]
Then $\widetilde{E}$ is a differential module over $E$,
$\widetilde{E}^{i,j}_1\cong \coh^{i+j}(\gr_i A, \gr M)$,
and $\widetilde{E}$ converges to $\coh^*(A,M)$ as an $\coh^*(A,k)$-module
(see~\cite[Theorem 4]{M}).


\subsection{Hochschild cohomology of a filtered algebra}
\label{subsec:May-HH}

There are similar spectral sequences for the Hochschild cohomology
of a filtered (not necessarily augmented) algebra $A$.
We assume the filtration on $A$ is {\em increasing} (the case of
a decreasing filtration is similar).
This induces an increasing filtration on the bar resolution
$P_{\bu}$ of $A$ as an $A$-bimodule (equivalently, $A^e$-module 
where $A^e=A\ot A^{op}$). 
Let $C^{\bu}(A,A) = \Hom_{A^e}(P_{\bu}, A)$, a filtered
vector space with {\em decreasing} filtration:
\[
    F^i C^n (A,A)~ =~ \{ f\in\Hom_{A^e}(P_n,A): 
     f(F_j P_n) \subset F_{j-i} A \text{ for all } j\}.
\]
This filtered complex gives rise to a spectral sequence
(see~\cite[5.4.1 and 5.5.1]{W} or~\cite{McC}).
Similar to the construction given in the augmented algebra setting above,
\[
    E^{i,j}_0 ~= ~\frac{F^iC^{i+j} (A,A)}{F^{i+1} C^{i+j}(A,A)},
   \qquad E^{i,j}_1 ~= ~\coh^{i+j}(E^{i,*}_0)~\cong~ \HH^{i+j}(\gr_i A, \gr A).
\]
Moreover if $B$ is an $A$-bimodule, the
filtration on $A$ induces a filtration on $B$ and 
on $C^{\bu}(A,B) = \Hom_{A^e}(P_{\bu},B)$:
\[
   F^i C^n (A,B) ~=~ \{ f\in\Hom_{A^e}(P_n,B): 
        f(F_\ell P_n) \subset F_{\ell-i} B \text{ for all } \ell \}.
\]
There is
similarly a spectral sequence $\widetilde{E}$ associated to
this filtration,
$\widetilde{E}$ is a differential module over $E$,
$\widetilde{E}_1^{i,j} \cong \HH^{i+j}(\gr_i A, \gr B)$,
and $\widetilde{E}$ converges to $\HH^*(A,B)$ as an $\HH^*(A)$-module.


\subsection{Cohomology of an augmented crossed product}
\label{subsec:cohomology-smash}

Assume $H$ is a finite dimensional Hopf algebra over a field $k$ with antipode $S$, $R$ is an algebra for which a crossed product $R\#_{\sigma} H$ exists, and $R$ and $R\#_{\sigma}H$ are augmented compatibly with the counit on $H$. We explicitly construct a multiplicative spectral sequence converging to $\Ext^*_{R\#_{\sigma} H}(k,k)$. This spectral sequence is known; see e.g., similar 
constructions of~\cite[\S 5]{ginzburg-kumar93} and~\cite[\S 2.10]{SV}. We will however need some additional structure on the spectral sequences for applications, 
and we give details of the construction
to make these additional structures more transparent. 

For any $R\#_{\sigma} H$-modules $U$ and $V$, there is a right $H$-module structure
on $\Hom_R(U,V)$ given by
\begin{equation}\label{eqn:right-action}
   (f\cdot h) (u)~=~\sum S(h_1)\cdot (f(h_2\cdot u))
\end{equation} 
for all $f\in\Hom_R(U,V)$, $h\in H$, and $u\in U$. 
A calculation shows that $f\cdot h$ is indeed an $R$-module
homomorphism for all $h\in H$ and $f\in \Hom_R(U,V)$ since $U,V$
are $R\#_{\sigma} H$-modules. Another calculation shows that this gives an $H$-module structure
to $\Hom_R(U,V)$. 

Let $P_{\bu}$ be a projective resolution of $k$ as an $H$-module.
Let $Q_{\bu}$ be a projective resolution of $k$ as an $R\#_{\sigma} H$-module.
Since $R\#_{\sigma} H$ is free as an $R$-module, $Q_{\bu}$ restricts to a projective
resolution of $k$ as an $R$-module. Let $M$ be an $R\#_{\sigma} H$-module, which can be considered as an $R$-module. 

For all $i,j\geq 0$,
let 
\[
    C^{i,j}(M)~=~\Hom_H(P_i, \Hom_R(Q_j,M)) .
\]
Letting $\coh '$ denote homology with respect to the horizontal differentials
and $\coh ''$ with respect to the vertical differentials,
the spectral sequence associated to this double complex has second page
\[
    E_2^{*,*}(M) ~= ~\coh ' \coh '' (C^{i,j}) 
        ~ \cong~  \coh ' (\Hom_H (P_{\bu}, \coh^*(R,M)) 
   ~ \cong~  \coh^*(H, \coh^*(R,M)) .
\]
We wish to compare with the spectral sequence obtained by reversing 
the roles of $i$ and $j$.
To that end, we will need the following lemma.

\begin{lemma}
For each $j$, the $H$-module
$\Hom_R (Q_j , M)$ is projective.
\end{lemma}

\begin{proof}
For each $j$, the $H$-module $Q_j$, under restriction from $R\#_{\sigma} H$,
is projective, since $R\#_{\sigma} H$ is free as an $H$-module. 
In order to see that $\Hom_R(Q_j, M)$ is also projective as an $H$-module,
it suffices to prove that $\Hom_R(R\#_{\sigma} H , M)$ is a projective $H$-module,
since $\Hom_R$ and taking cohomology are additive. 
First notice that 
\[
    \Hom_R (R\#_{\sigma} H , M)~\cong~ \Hom_k(H, M)
\]
as vector spaces since an $R$-module homomorphism from $R\#_{\sigma} H$ to $M$
is determined by its values on $H$.
The right $H$-action on $\Hom_k(H,M)$ induced by this isomorphism
is also given by formula~\eqref{eqn:right-action}.
We claim that $\Hom_k(H,M)\cong M\ot H^{\vee}$ as right $H$-modules,
where $M$ is a right $H$-module via its left $H$-module structure
and the antipode $S$, and $H^{\vee} = \Hom_k(H,k)$ is a right
$H$-module via left multiplication of $H$ on itself:
\[
   m\cdot h ~= ~S(h)\cdot m \ \ \ \mbox{ and } \ \ \
     (f\cdot h) (\ell) ~= ~f(h\ell),
\]
for all $m\in M$, $h,\ell\in H$, and $f\in H^{\vee}$.
The isomorphism $\Hom_k(H,M)\stackrel{\sim}{\longrightarrow}
M\ot H^{\vee}$ is given by sending $m\ot f$ to the function 
taking $\ell$ to $f(\ell) m$, for all $m\in M$, $f\in\Hom_k (H,M)$,
and $\ell\in H$. 
Since $H$ is a Frobenius algebra, $H^{\vee}\cong H$ as an $H$-module
(an isomorphism $H\stackrel{\sim}{\longrightarrow} H^{\vee}$ is given
by sending $h$ to the function that takes $\ell$ to $\phi(h\ell)$,
where $\phi$ is an integral of the dual Hopf algebra to $H$). 
It follows that $M\ot H^{\vee} \cong M\ot H$ as $H$-modules.
Since $H$ is free as a right $H$-module,
the $H$-module $M\ot H$ is projective~\cite[Proposition 3.1.5]{Benson1}.
\end{proof}

Now, reversing the roles of $i$ and $j$, there is another spectral sequence $\widetilde{E}^{*,*}(M)$ with second page 
\[
   \widetilde{E}^{*,*}_2(M) \cong \coh '' \coh ' (C^{i,j})~ = ~
     \coh '' (\coh^*(H, \Hom_R(Q_{\bu} , M)).
\]
Since $\Hom_R(Q_j, M)$ is a projective $H$-module for each $j$,
it is also injective, and so
\[
   \coh^*(H, \Hom_R(Q_{\bu}, M))  ~\cong~  \coh^0 (H, \Hom_R(Q_{\bu}, M))
    ~\cong~  \Hom_R(Q_{\bu} , M) ^H
   ~\cong~  \Hom_{R\#_{\sigma} H} (Q_{\bu}, M) .
\]
Thus we see that $\widetilde{E}_2(M) \cong \coh^*(R\#_{\sigma} H , M)$.

The above discussion yields the following result.

\begin{lemma} 
\label{lem:ss-mod} 
Let $M$ be an $R \#_{\sigma} H$-module. There is a spectral sequence $E^{*,*}(M)$ with
\begin{equation*}
    E_2^{p,q}(M)= \coh^p(H,\coh^q(R,M)) \Longrightarrow \coh^{p+q}(R\#_{\sigma} H , M) .
\end{equation*}
\end{lemma}

\begin{proof}
By construction, $E^{*,*}(M)$ is a first quadrant spectral sequence,
and so the standard filtration for a double complex is bounded.
It follows that $E^{*,*}(M)$ converges (see, e.g.,~\cite[Theorem~5.5.1]{W}).
The rest of the statement results from our discussion above.
\end{proof}

Next we will show that when $M =k$, the spectral sequence $E^{*,*}(k)$ 
is multiplicative, and further that $E^{*,*}(M)$ is a differential
bigraded module over $E^{*,*}(k)$.
We will prove the latter statement; the former is a special case.

For all $i,j\geq 0$, let
\[
   B^{i,j} ~=~ \Hom_H(P_i, \Hom_R(Q_{\bu}, Q_{\bu})_j) ,
\]
which is quasi-isomorphic to $C^{\bu,\bu}$ when $M=k$.
There is an action of $B^{\bu,\bu}$ on $C^{\bu,\bu}$ given as follows.
Let $\Delta: P_{\bu}\rightarrow P_{\bu}\ot P_{\bu}$
be a diagonal chain map.
Let
\[
   \mu : \Hom_R(Q_{\bu}, M) \ot \Hom_R(Q_{\bu},Q_{\bu})
       \rightarrow \Hom_R(Q_{\bu},M)
\]
denote composition of functions.

Let $f\in C^{i,j}$ and $g\in B^{i',j'}$, and define $f\cdot g\in C^{i+i',j+j'}$ by
\begin{equation}\label{eqn:fdotg}
   f\cdot g ~= ~\mu (f\ot g) \Delta .
\end{equation}

\begin{prop}
\label{prop:dgmod-cplx-coh}
For any $R \#_{\sigma} H$-module $M$, the bicomplex $C^{\bu,\bu} = \Hom_H(P_{\bu}, \Hom_R(Q_{\bu},M))$ 
is a differential bigraded module over $B^{\bu,\bu} = \Hom_H(P_{\bu}, \Hom_R(Q, Q)_{\bu})$ under the action defined by equation~(\ref{eqn:fdotg}). 
\end{prop}

\begin{proof}
Let $f\in C^{i,j}$ and $g\in B^{i',j'}$. It suffices for us to check that the following equation holds: 
\begin{equation}
\label{eqn:delta-f-g}
   \delta(f\cdot g) ~=~ \partial (f) \cdot g + (-1)^{i+j} g \cdot \delta(g) ,
\end{equation}
where $\delta$ is the differential on $C^{\bu,\bu}$ and 
$\partial$ is the differential on $B^{\bu,\bu}$.

To verify \eqref{eqn:delta-f-g}, we expand the left side of~\eqref{eqn:delta-f-g}, letting
$d_Q^*(f\cdot g)$ denote composition of the function $f\cdot g$
with $d_Q$, to obtain
\begin{eqnarray*}
      \delta(f\cdot g) &~ = ~& d_Q^*(f\cdot g) + 
   (-1)^{i+j+i'+j'+1} (f\cdot g) d_P \\
    &~=~& d_Q^* \mu (f\ot g) \Delta + (-1)^{i+j+i'+j'+1} \mu 
     (f\ot g) \Delta d_p .
\end{eqnarray*}
Now $\Delta$ is a chain map, and $\Hom_R(Q_{\bu},Q_{\bu})$ is a
differential graded algebra with differential graded module
$\Hom_R(Q_{\bu},k)$, so letting $\delta_Q$ denote the
differential on $\Hom_R(Q_{\bu},Q_{\bu})$, the above expression
is equal to 
\begin{eqnarray*}
  &&\hspace{-1cm} \mu(d_Q^* f\ot g)\Delta + (-1)^{i+j} \mu (f\ot \delta_Q g)\Delta 
    + (-1)^{i+j+i'+j'+1} \mu(f\ot g) (d_P\ot 1
     + 1\ot d_P ) \Delta \\
  &~=~&  \mu (d_Q^* f\ot g + (-1)^{i+j} f\ot \delta_Q g 
    + (-1)^{i+j+1} fd_P\ot g + (-1)^{i+j+i'+j'+1} f\ot gd_P ) \Delta .
\end{eqnarray*}
On the other hand, 
\begin{eqnarray*}
 && \hspace{-1cm} \partial(f) \cdot g  + (-1)^{i+j} f\cdot \delta (g)\\
  &~=~&
   \mu ( (d_Q^* f + (-1)^{i+j+1} fd_P)\ot g ) \Delta 
  + (-1)^{i+j} \mu (f\ot (\delta_Q g + (-1)^{i'+j'+1} g d_P )) \Delta\\
 &~=~& \mu ( d_Q^* f \ot g + (-1)^{i+j+1} fd_P\ot g 
   + (-1)^{i+j} f \ot \delta_Q g + (-1)^{i+j +i'+j'+1} f\ot g d_p)\Delta ,
\end{eqnarray*}
and we see that the two expressions are equal, verifying~\eqref{eqn:delta-f-g}.
\end{proof}

We now see that from the definitions that each page $E^{*,*}_r(M)$ is a 
differential graded module over $E^{*,*}_r(k)$.
It follows that 
the action of $\coh^*(R\#_{\sigma} H , k)$ on $\coh^*(R\#_{\sigma} H,M)$ indeed arises
as a consequence of the above structure induced on the spectral sequences
corresponding to the bicomplexes $C^{\bu,\bu}$ and $B^{\bu,\bu}$:
These structures arise from the diagonal map $\Delta$ on $P_{\bu}$
and action given by composition on $\Hom_R(Q_{\bu},M)\ot \Hom_R(Q_{\bu},Q_{\bu})$. We conclude a needed result concerning spectral sequences for the cohomology of augmented smash products. 

\begin{prop}
\label{prop:dgmod-ss-coh}
Let $H$ be a finite dimensional Hopf algebra over a field $k$ and let $R$ be an $H$-module algebra. Assume $R$ and $R \#_{\sigma}H$ are augmented compatibly with the counit on $H$. For any $R \#_{\sigma} H$-module $M$, let $E^{*,*}(M)$ denote the spectral sequence of 
Lemma~\ref{lem:ss-mod}. 
Then $E^{*,*}(M)$ is a differential bigraded module over the dg algebra $E^{*,*}(k)$, and
the following properties hold:
\begin{enumerate}
 \item $\Ext^*_R(k,k)$ is a graded $H$-module algebra. 
 \item $\Ext^*_R(k,M)$ is a graded module over $\Ext^*_R(k,k) \# H$. 
 \end{enumerate}

\end{prop}

\begin{proof}
The first statement is an immediate consequence of 
Proposition~\ref{prop:dgmod-cplx-coh} by the above discussion. 
The rest involves the $H$-action. Note that 
from the definition~\eqref{eqn:right-action} of the $H$-action, 
we see that $\mu$ is an $H$-module map, that is
$
\mu((f \ot g) \cdot h) ~=~ (\mu \cdot h)(f \ot g).
$
Further, the action of $H$ commutes with the differentials.
Passing to homology, the induced map 
$$\coh^*(\mu) : \coh^*(\Hom_R(Q_{\bu}, M)) \ot \coh^*(\Hom_R(Q_{\bu},Q_{\bu})) \rightarrow \coh^*(\Hom_R(Q_{\bu},M))$$
is an $H$-module map. We see that 
the following diagram commutes, where the bottom row is given by Yoneda product: 
\[ 
\begin{xy}*!C\xybox{
\xymatrixcolsep{2pc}
\xymatrix{
 & \coh^*(\Hom_R(Q_{\bu}, M)) \ot \coh^*(\Hom_R(Q_{\bu},Q_{\bu})) \ar[rr]^-{\coh^*(\mu)} \ar@{=}[d] && \coh^*(\Hom_R(Q_{\bu},M)) \ar@{=}[d] \\
& \Ext^*_R(k,M) \otimes \Ext^*_R(k,k) \ar[rr] && \Ext^*_R(k,M)
}}
\end{xy} 
\]
The diagram is commutative by a standard argument:
In the correspondence of elements of $\Ext$ with
generalized extensions, Yoneda composition (or splice) of
generalized extensions corresponds with composition of chain maps
representing the generalized extensions. 
Moreover, the $H$-action is compatible with the cup/Yoneda products
by their definitions. 

Since $\Ext^*_R(k,M) \cong \coh^*(\Hom_R(Q_{\bu}, M))$ and 
$\Ext^*_R(k,k)\cong \coh^*(\Hom_R(Q_{\bu},Q_{\bu}))$,  
statement (ii) follows. Statement (i) is just a special case by letting $M=k$. 
\end{proof}


\subsection{Hochschild cohomology of a smash product}
\label{subsec:hochschild}

Let $H$ be a Hopf algebra  with bijective antipode and let $R$ be an $H$-module algebra.
Let $M$ be an $R\# H$-bimodule. 
\c{S}tefan~\cite{Stefan} constructed a spectral sequence
\begin{equation}
    \coh^p( H, \HH^q(R, M)) \Longrightarrow \HH^{p+q}(R\# H , M) ,
\end{equation}
and in fact he did this in the more general setting 
of a Hopf Galois extension.
Taking $M = R\# H$, one would like the spectral sequence to be multiplicative.
However, Negron~\cite{Negron15} pointed out that it is not multiplicative, and instead 
gave a different spectral sequence, in this special case 
of a smash product, that is multiplicative when $M = R\# H$.
We will show further that the techniques of~\cite{Negron15} imply that 
this alternative spectral sequence
is a differential bigraded module over that for $M=R\# H$. 
We first recall the construction of these spectral sequences from~\cite{Negron15}.

As in~\cite[Definition 3.1]{Negron15}, an $R$-bimodule $U$ is
{\bf $H$-equivariant} if it is an $H$-module for which the
structure maps $R\ot U\rightarrow U$ and $U\ot R\rightarrow U$
are $H$-module homomorphisms. 
As is pointed out in~\cite{Negron15}, Kaygun~\cite[Lemma~3.3]{Kaygun}
showed that $H$-equivariant $R$-bimodules are in fact modules
for a particular smash product $R^e\# H$. 
An $H$-equivariant $R^e$-complex is defined similarly.

For any $H$-equivariant $R$-bimodule $U$ and 
$R\# H$-bimodule $V$, 
there is a right $H$-module structure on $\Hom_{R^e}(U,V)$ given
in~\cite[Definition~4.1]{Negron15}:
\[
      (f\cdot h) (u) ~= ~\sum S(h_1) f(h_2\cdot u) h_3
\]
for all $f\in \Hom_{R^e}(U,V)$, $h\in H$, and $u\in U$.

Let $L_{\bu}$ be a projective resolution of $k$ as an $H$-module.
Let $K_{\bu}$ be an $H$-equivariant $R$-bimodule resolution of $R$.
(Existence is discussed in~\cite{Negron15}; there it is assumed
that $K_{\bu}$ is free over $R^e$ on a graded base space
$\overline{K}_{\bu}$ that is an $H$-submodule.) 
By~\cite[Corollary~4.4]{Negron15}, there is a graded vector
space isomorphism 
\begin{equation*}
    \HH^*(R\# H, M)~ \cong~ \coh^*( \Hom_H (L_{\bu} , \Hom_{R^e}(K_{\bu}, M)).
\end{equation*}
Standard filtrations of this bicomplex result in a 
spectral sequence where
\begin{equation}
\label{eqn:bimodule-ss}
    \coh^p( H, \HH^q(R, M)) \Longrightarrow \HH^{p+q}(R\# H , M) .
\end{equation}

By~\cite[Corollary~6.8]{Negron15}, when $M = R\# H$,
the spectral sequence~(\ref{eqn:bimodule-ss}), as
constructed above, is multiplicative.
The proof of~\cite[Corollary~6.8]{Negron15} begins 
with~\cite[Proposition~6.1]{Negron15}, and we extend this to the module 
setting next.
The action of $\Hom_{R^e}(K_{\bu}, R\# H)$
on $\Hom_{R^e}(K_{\bu}, M)$ is defined via a diagonal map
$\omega : K_{\bu}\rightarrow K_{\bu}\ot_R K_{\bu}$, analogously to the cup
product defined in~\cite[Section 5]{Negron15}, making it a differential
graded module.

\begin{prop}
\label{prop:dgmod-cplx-hochschild}
For any $R\# H$-bimodule $M$, the complex $\Hom_{R^e}(K_{\bu}, M)$ 
is a differential graded module over $\Hom_{R^e}(K_{\bu}, R\# H)$,
with compatible $H$-action. 
\end{prop}

\begin{proof}
This is essentially the proof of~\cite[Proposition 6.1]{Negron15},
checking compatibility with the $H$-action. 
\end{proof}

We now conclude a needed result concerning spectral sequences for the  Hochschild cohomology of smash products. 

\begin{prop}
\label{prop:dgmod-ss-hochschild}
Let $H$ be a Hopf algebra with bijective antipode and let $R$ be an $H$-module algebra.
For any $R \# H$-bimodule $M$, let $E^{*,*}(M)$ denote the spectral sequence~\eqref{eqn:bimodule-ss} above. 
Then $E^{*,*}(M)$ is a differential bigraded module over the dg algebra $E^{*,*}(R \# H)$
and the following properties hold:
\begin{enumerate}
 \item $\Ext^*_{R^e}(R,R \#H)$ is a graded $H$-module algebra. 
 \item $\Ext^*_{R^e}(R,M)$ is a graded module over $\Ext^*_{R^e}(R,R\# H) \# H$. 
 \end{enumerate}
\end{prop}

\begin{proof}
A similar calculation to that of Proposition~\ref{prop:dgmod-cplx-coh} 
shows directly that $\Hom_H(L_{\bu}, \Hom_{R^e}(K_{\bu},M))$ is a
differential bigraded module over $\Hom_H(L_{\bu},\Hom_{R^e}(K_{\bu},R\# H))$. 
Alternatively, we give an argument similar to that in~\cite{Negron15} 
for the first statement.
That is, we use Proposition~\ref{prop:dgmod-cplx-hochschild}
combined with properties of the resolution $L_{\bu}$ of $k$ as an $H$-module:
For any right $H$-module $U$, let $U^{\uparrow}$ denote the induced $H^e$-module $U \ot_H H^e$ from $U$, where $H^e$ is viewed as a left $H$-module with the action given by $h \mapsto \sum h_1 \ot S(h_2)$.
Let $L_{\bu}^{\uparrow}$ denote the complex $L_{\bu}$ for which each
module has been induced in this way to an $H^e$-module.
By \cite[Theorem~3.5]{Negron15}, $K_{\bu} \# L_{\bu}^{\uparrow}$ is a projective $R \# H$-bimodule resolution of $R \# H$. Furthermore, for any $R\# H$-bimodule $M$, there is a natural isomorphism of chain complexes \cite[Theorem~4.3]{Negron15},
$$\Xi: \Hom_{(R\#H)^e}(K_{\bu} \# L_{\bu}^{\uparrow},M) \xrightarrow{\cong} \Hom_H(L_{\bu},\Hom_{R^e}(K_{\bu},M)).$$
When $M=R \#H$, it follows from \cite[Theorem~6.5]{Negron15} that $\Xi$ is an isomorphism of dg algebras. Moreover, similar to the argument given in the proof of \cite[Theorem~6.5]{Negron15},
one can check that the following diagram commutes:
\[
\begin{xy}*!C\xybox{
\xymatrixcolsep{2pc}
\xymatrix{
 \Hom(K_{\bu} \# L_{\bu}^{\uparrow},R\#H) \ot \Hom(K_{\bu} \# L_{\bu}^{\uparrow},M) \ar[r]^-{\Xi \ot \Xi} \ar[d]^{\rho} & \Hom(L_{\bu},\Hom(K_{\bu},R\#H)) \ot \Hom(L_{\bu},\Hom(K_{\bu},M)) \ar[d]^{\rho '} \\
 \Hom(K_{\bu} \# L_{\bu}^{\uparrow},M) \ar[r]^-{\Xi} & \Hom(L_{\bu}, \Hom(K_{\bu},M)) , 
}}
\end{xy}
\]
where $\rho$ is the module action resulting from the diagonal map
on $K_{\bu}\# L_{\bu}^{\uparrow}$ given in~\cite[Proposition~6.4]{Negron15}
and $\rho '$ is a similarly defined module action. 
Thus the action on the left corresponds to the action on the right.
The action on the left arises from a diagonal map on 
$K_{\bu}\# L_{\bu}^{\uparrow}$ and so it is a differential
bigraded module structure.
Therefore the same is true on the right. 
The corresponding spectral sequences are induced by the row and column filtrations on the first quadrant double complexes $\Hom_H(L_{\bu},\Hom_{R^e}(K_{\bu},R\#H))$ and $\Hom_H(L_{\bu},\Hom_{R^e}(K_{\bu},M))$. Thus, $E^{*,*}(M)$ is a differential bigraded module over the dg algebra $E^{*,*}(R \# H)$.

Now taking homology in Proposition~\ref{prop:dgmod-cplx-hochschild}, we see that $\Ext^*_{R^e}(R,M) \cong \coh^*(\Hom_{R^e}(K_{\bu}, M))$ is a graded module over $\Ext^*_{R^e}(R,R\# H) \cong \coh^*(\Hom_{R^e}(K_{\bu}, R\# H))$, and the module structure is compatible with the $H$-action. Hence $\Ext^*_{R^e}(R,M)$ is a graded module over the smash product $\Ext^*_{R^e}(R,R\# H) \# H$. This proves (ii). Statement (i) is a special case by letting $M=R\# H$. 
\end{proof}


\section*{Acknowledgments}
Part of this research work was done during the second author's visit to Texas A\&M University in May 2018. He is grateful for the invitation of the third author and thanks Texas A\&M University for its hospitality. All the authors also thank Mathematisches Forschungsinstitut at Oberwolfach for hosting the mini-Workshop: Cohomology of Hopf Algebras and Tensor Categories in March 2019.


\end{document}